\documentclass[reqno]{amsart}
\usepackage{amsmath}
\usepackage{amsthm,amscd,amssymb,amsfonts, amsbsy}
\usepackage{latexsym}
\usepackage{mathrsfs}
\usepackage{pxfonts}
\usepackage{enumerate}

\date{\today}

\numberwithin{equation}{section}

\theoremstyle{plain}
\newtheorem{theorem}{Theorem}[section]
\newtheorem{lemma}{Lemma}[section]
\newtheorem{corollary}{Corollary}[section]

\theoremstyle{definition}
\newtheorem{definition}{Definition}[section]

\newtheorem*{A0}{(A0)}
\newtheorem*{A1}{(A1)}
\newtheorem*{A2}{(A2)}
\newtheorem*{D}{(D)}
\newtheorem*{A}{(A3 $(\gamma)$)}

\newtheorem*{acknowledgment}{Acknowledgment}

\theoremstyle{remark}
\newtheorem{remark}{Remark}[section]

\newcommand{\dv}{\operatorname{div}}

\newcommand{\supp}{\operatorname{supp}}
\newcommand{\dist}{\operatorname{dist}}
\newcommand{\diam}{\operatorname{diam}}
\newcommand{\tr}{\operatorname{tr}}

\newcommand{\Lip}{\operatorname{Lip}}
\newcommand{\sgn}{\operatorname{sgn}}

\renewcommand{\vec}[1]{\boldsymbol{#1}}

\newcommand{\VMO}{\mathrm{VMO}}

\newcommand{\bR}{\mathbb R}

\newcommand\cA{\mathcal{A}}
\newcommand\cB{\mathcal{B}}

\newcommand\cF{\mathcal{F}}
\newcommand\cG{\mathcal{G}}
\newcommand\cK{\mathcal{K}}

\newcommand\cM{\mathcal{M}}

\newcommand\cQ{\mathcal{Q}}

\newcommand\cU{\mathcal{U}}

\newcommand\sB{\mathscr{B}}

\newcommand\sL{\mathscr{L}}

\providecommand{\set}[1]{\{#1\}}
\providecommand{\Set}[1]{\left\{#1\right\}}
\providecommand{\bigset}[1]{\bigl\{#1\bigr\}}

\providecommand{\abs}[1]{\lvert#1\rvert}
\providecommand{\Abs}[1]{\left\lvert#1\right\rvert}
\providecommand{\bigabs}[1]{\bigl\lvert#1\bigr\rvert}

\providecommand{\norm}[1]{\lVert#1\rVert}

\providecommand{\bignorm}[1]{\bigl\lVert#1\bigr\rVert}

\renewcommand{\epsilon}{\varepsilon}
\renewcommand{\qedsymbol}{$\blacksquare$}

\begin{document}
\title[Stokes systems]{The Green function for the Stokes system with measurable coefficients}

\author[J. Choi]{Jongkeun Choi}
\address[J. Choi]{Department of Mathematics, Korea University, Seoul 02841, Republic of Korea}
\email{jongkeun\_choi@korea.ac.kr}
\thanks{}

\author[K.-A. Lee]{Ki-Ahm Lee}
\address[K.-A. Lee]{Department of Mathematical Sciences, Seoul National University, Seoul 151-747, Republic of Korea \& Center for Mathematical Challenges, Korea Institute for Advanced Study, Seoul 130-722, Republic of Korea}
\email{kiahm@snu.ac.kr}
\thanks{}

\subjclass[2010]{Primary 35J08, 35J57, 35R05}
\keywords{Stokes system; Green function; $L^q$-estimate; $\mathrm{VMO}$ coefficients; Reifenberg flat domain}


\begin{abstract}
We study the Green function 
for the stationary Stokes system with bounded measurable coefficients in a bounded Lipschitz domain $\Omega\subset \mathbb{R}^n$, $n\ge 3$.
We construct the Green function in $\Omega$ under the condition $(\bf{A1})$ that weak solutions of the system enjoy interior H\"older continuity.
We also prove that $(\bf{A1})$ holds, for example, when the coefficients are $\mathrm{VMO}$.
Moreover, we obtain the global pointwise estimate for the  Green function under the additional assumption $(\bf{A2})$ that weak solutions of Dirichlet problems are locally  bounded up to the boundary of the domain.
By proving a priori $L^q$-estimates for  Stokes systems with $\mathrm{BMO}$ coefficients on a Reifenberg domain, we verify that $(\bf{A2})$ is satisfied when the coefficients are $\mathrm{VMO}$ and $\Omega$ is a bounded $C^1$ domain.
\end{abstract}

\maketitle

\section{Introduction}
We consider the Dirichlet boundary value problem for the stationary Stokes system
\begin{equation}		\label{1j}
\left\{
\begin{aligned}
\sL\vec u+Dp=\vec f +D_\alpha \vec f_\alpha\quad &\text{in }\, \Omega,\\
\dv \vec u=g \quad &\text{in }\, \Omega,\\
\vec u=0 \quad &\text{on }\, \partial \Omega,
\end{aligned}
\right.
\end{equation}
where $\Omega$ is a domain in $\bR^n$.
Here, $\sL$ is an  elliptic operator of the form
\[
\sL\vec u=-D_\alpha(A_{\alpha\beta}D_\beta \vec u), 
\]
where the coefficients $A_{\alpha\beta}=A_{\alpha\beta}(x)$ are $n\times n$ matrix valued functions on  $\bR^n$ with entries $a_{\alpha\beta}^{ij}$ that satisfying the strong ellipticity condition; i.e., there is a constant $\lambda\in (0,1]$ such that for any $x\in \bR^n$ and $\vec \xi, \, \vec \eta\in \bR^{n\times n}$, we have 
\begin{equation}		\label{ubc}	
\lambda\abs{\vec \xi}^2\le a_{\alpha\beta}^{ij}(x)\xi^j_\beta \xi^i_\alpha, \qquad \bigabs{a_{\alpha\beta}^{ij}(x)\xi^j_\beta \eta^i_\alpha}\le \lambda^{-1}\abs{\vec \xi}\abs{\vec\eta}.
\end{equation}
We do not assume that the coefficients $A_{\alpha\beta}$ are symmetric.
The adjoint operator $\sL^*$ of $\sL$ is given by 
\[
\sL^*\vec u=-D_\alpha(A_{\beta \alpha}(x)^{\tr} D_\beta \vec u).
\]
We remark that the coefficients of $\sL^*$  also satisfy \eqref{ubc} with the same constant  $\lambda$.
There has been some interest in studying boundary value problems for Stokes systems with bounded coefficients; see, for instance, Giaquinta-Modica \cite{MR0641818}.
They obtained various interior and boundary estimates for both linear and nonlinear systems of the type of the stationary Navier-Stokes system.

Our first focus  is to study of the Green function for the  Stokes system with $L^\infty$ coefficients in a bounded Lipschitz domain $\Omega\subset \bR^n$, $n\ge 3$.
More precisely, we consider a pair $(\vec G(x,y), \vec \Pi(x,y))$, where $\vec G(x,y)$ is an $n\times n$ matrix valued function and $\vec \Pi(x,y)$ is an $n\times1$ vector valued function on $\Omega\times \Omega$, satisfying
\[
\left\{
\begin{aligned}
\sL_x\vec G(x,y)+D_x \vec \Pi(x,y)=\delta_y(x)\vec I &\quad \text{in }\, \Omega,\\
\dv_x\vec G(x,y)=0 &\quad \text{in }\, \Omega,\\
\vec G(x,y)=0& \quad \text{on }\, \partial \Omega.
\end{aligned}
\right.
\]
Here, $\delta_y(\cdot)$ is Dirac delta function concentrated at $y$ and $\vec I$ is the $n\times n$ identity matrix. 
See Definition \ref{0110.def} for the precise definition of the Green function.
We prove that if  weak solutions  of either
\[
\sL\vec u+Dp=0, \quad \dv \vec u=0 \quad \text{in }\, B_R
\]
or 
\[
\sL^*\vec u+Dp=0, \quad \dv \vec u=0 \quad \text{in }\, B_R
\]
satisfy the following De Giorgi-Moser-Nash type estimate 
\begin{equation}		\label{0211.eq1}
[\vec u]_{C^\mu(B_{R/2})}\le CR^{-n/2-\mu}\norm{\vec u}_{L^2(B_R)},
\end{equation}
then the Green function $(\vec G(x,y), \vec \Pi(x,y))$ exists and satisfies a natural growth estimate near the pole; see Theorem \ref{1226.thm1}.
It can be shown, for example, that if  the coefficients of $\sL$ belong to the class of $\VMO$ (vanishing mean oscillations), then the  interior H\"older estimate \eqref{0211.eq1} above holds; see Theorem \ref{0110.thm1}.
Also, we are interested in the following global pointwise estimate for the Green function: there exists a positive constant $C$ such that 
\begin{equation}		\label{0304.eq2}
\abs{\vec G(x,y)}\le C\abs{x-y}^{2-n}, \quad \forall x,\, y\in \Omega, \quad x\neq y.
\end{equation}
If we assume further that the operator $\sL$ has the property that the weak solution of 
\[
\left\{
\begin{aligned}
\sL\vec u+Dp=\vec f &\quad \text{in }\, \Omega,\\
\dv{\vec u}=g &\quad \text{in }\, \Omega,\\
\vec u=0 &\quad \text{on }\, \partial \Omega,
\end{aligned}
\right.
\]
is locally bounded up to the boundary, then we obtain  the pointwise estimate \eqref{0304.eq2} of the Green function.
This local boundedness condition $(\bf{A2})$ is satisfied when  the coefficients of $\sL$ belong to the class of $\VMO$ and $\Omega$ is a bounded $C^1$ domain.
To see this, we employ the standard localization method and the global $L^q$-estimate for the Stokes system with Dirichlet boundary condition, which is our second focus in this paper.

Green functions for the linear equation and system  have been studied by many authors. 
In \cite{MR0161019}, Littman-Stampacchia-Weinberger obtained the pointwise estimate of the Green function for elliptic equation.
Gr\"uter-Widman \cite{MR0657523} proved existence and uniqueness of the Green function for elliptic equation, and the corresponding results for elliptic systems with continuous coefficients were obtained in \cite{MR1354111,MR0894243}.
Hofmann-Kim proved the existence of Green functions for elliptic systems with variable coefficients on any open domain.
Their methods are general enough to allow the coefficients to be $\VMO$.
For more details, see \cite{MR2341783}.
We also refer the reader to \cite{MR2718661, MR2763343} and references therein for the study of Green functions for elliptic systems.
Regarding the study of the Green function for the  Stokes system with the Laplace operator, we refer the reader to \cite{MR2465713,MR2718661}.
In those papers, the authors obtained the global pointwise estimate \eqref{0304.eq2} for the Green function on a three dimensional Lipschitz domain.
Mitrea-Mitrea \cite{MR2763343} established regularity properties of the Green function for the Stokes system with Dirichlet boundary condition in a  two or three dimensional Lipschitz domain.
Recent progress may be found in the article of Ott-Kim-Brown \cite{MR3320459}.
This work  includes a construction of the Green function with mixed boundary value problem for the Stokes system in two dimensions.

Our second focus in this paper is the global $L^q$-estimates for the Stokes systems of divergence form with  the Dirichlet boundary condition.
As mentioned earlier, the $L^q$-estimate for the Stokes system is the key ingredient in establishing the global pointwise estimate for the Green function. 
Moreover, the study of the regularity of solutions to the Stokes system plays an essential role in the mathematical theory of viscous fluid flows governed by the Navier-Stokes system.
For this reason, the $L^q$-estimate for the Stokes system with the Laplace operator was discussed in many papers. 
We refer the reader to Galdi-Simader-Sohr \cite{MR1313554}, Maz'ya-Rossmann \cite{MR2321139}, and  references therein.
Recently,  estimates in Besov spaces for the  Stokes system are obtained by Mitrea-Wright \cite{MR2987056}.
In this paper, we consider the $L^q$-estimates for Stokes systems with variable coefficients in non-smooth domains.
More precisely, we prove that if the coefficients of $\sL$ 
have small bounded mean oscillations
on a Reifenberg flat domain $\Omega$, then the solution $(\vec u, p)$ of the problem \eqref{1j} satisfies the following $L^q$-estimate:
\begin{equation*}	
\norm{p}_{L^q(\Omega)}+\norm{D\vec u}_{L^q(\Omega)}\le  C\big(\norm{\vec f}_{L^q(\Omega)}+\norm{\vec f_\alpha}_{L^q(\Omega)}+\norm{g}_{L^q(\Omega)}\big).
\end{equation*}
Moreover, we obtain the solvability in Sobolev space for the systems on a bounded Lipschitz domain.
It has been studied by many authors that the $L^q$-estimates for elliptic and parabolic systems with variable coefficients on a Reifenberg flat domain.
We refer the reader to Dong-Kim \cite{MR2835999, MR3013054} and Byun-Wang \cite{MR2069724}.
In particular, in \cite{MR2835999}, the authors proved $L^q$-estimates for divergence form higher order systems with partially BMO coefficients on a Reifenberg flat domain.
Their argument is based on mean oscillation estimates and $L^\infty$-estimates combined with  the measure theory on the ``crawling of ink spots" which can be found in  \cite{MR0563790}.
We mainly follow the arguments in \cite{MR2835999}, but the technical details are different due to the pressure term $p$.
The presence of the pressure term $p$ makes the argument more involved.

The organization of the paper is as follows.
In Section \ref{sec_mr}, we introduce some notation and state our main theorems, including the existence and global pointwise estimates for Green functions, and their proofs are presented in Section \ref{1006@sec1}. Section \ref{sec_es} is devoted to the study of the $L^q$-estimate for the Stokes system with the Dirichlet boundary condition.
In Appendix, we provide some technical lemmas.

\section{Main results}		\label{sec_mr}
Before we state our main theorems,   we introduce some necessary notation.
Throughout the article, we use $\Omega$ to denote a bounded domain in $\bR^n$, where $n\ge 2$. 
For any $x=(x_1,\ldots,x_n)\in \Omega$ and $r>0$,we write $\Omega_r(x)=\Omega \cap B_r(x)$, where $B_r(x)$ is the usual Euclidean ball of radius $r$ centered at $x$.
We also denote 
$$
B_r^+(x)=\{y=(y_1,\ldots,y_n)\in B_r(x):y_1>x_1\}.
$$
We define $d_x=\dist(x,\partial \Omega)=\inf\set{\abs{x-y}:y\in \partial \Omega}$.
For a function $f$ on $\Omega$, we denote the average of $f$ in $\Omega$ to be
\[
(f)_\Omega=\fint_\Omega f\,dx.
\]
We use the notation 
\[
\sgn z=
\left\{
\begin{aligned}
z/\abs{z} &\quad \text{if }\, z\neq 0,\\
0 &\quad \text{if }\, z=0.
\end{aligned}
\right.
\]
For $1\le q\le \infty$, we define the space $L^q_0(\Omega)$ as the family of all functions $u\in L^q(\Omega)$ satisfying 
$(u)_\Omega=0$.
We denote by $W^{1,q}(\Omega)$ the usual Sobolev space and $W^{1,q}_0(\Omega)$ the closure of $C^\infty_0(\Omega)$ in $W^{1,q}(\Omega)$.
Let $\vec f,\, \vec f_\alpha\in L^q(\Omega)^n$ and $g\in L^q_0(\Omega)$.
We say that $(\vec u,p)\in W^{1,q}_0(\Omega)^n\times {L}^q_0(\Omega)$ is a weak solution of the problem 
\begin{equation}\tag{SP}\label{dp}
\left\{
\begin{aligned}
\sL \vec u+D p=\vec f+D_\alpha\vec f_\alpha &\quad \text{in }\, \Omega,\\
\dv \vec u=g &\quad \text{in }\, \Omega,
\end{aligned}
\right.
\end{equation}
if we have 
\begin{equation}		\label{1218.eq0}
\dv \vec u=g\quad \text{in }\, \Omega
\end{equation}
 and 
\begin{equation*}
\int_\Omega A_{\alpha \beta}D_\beta \vec u\cdot D_\alpha \vec \varphi\,dx-\int_\Omega p\dv \vec\varphi\,dx
=\int_\Omega \vec f\cdot \vec \varphi\,dx-\int_\Omega \vec f_\alpha \cdot D_\alpha \vec \varphi\,dx\end{equation*}
for any $\vec \varphi\in C^\infty_0(\Omega)^n$.
Similarly, we say that $(\vec u,p)\in W^{1,q}_0(\Omega)^n\times {L}^q_0(\Omega)$ is a weak solution of the problem 
\begin{equation}\tag{SP$^*$}\label{dps}
\left\{
\begin{aligned}
\sL^* \vec u+D p=\vec f+D_\alpha\vec f_\alpha &\quad \text{in }\, \Omega,\\
\dv \vec u=g &\quad \text{in }\, \Omega,
\end{aligned}
\right.
\end{equation}
if we have \eqref{1218.eq0} and 
\begin{equation}		\label{1218.eq1}
\int_\Omega A_{\alpha \beta}D_\beta \vec \varphi\cdot D_\alpha \vec u\,dx-\int_\Omega p\dv \vec\varphi\,dx=\int_\Omega \vec f\cdot \vec \varphi\,dx-\int_\Omega \vec f_\alpha \cdot D_\alpha \vec \varphi\,dx
\end{equation}
for any $\vec \varphi\in C^\infty_0(\Omega)^n$.

\begin{definition}[Green function]		\label{0110.def}
Let $\vec G(x,y)$ be an $n\times n$ matrix valued function and $\vec \Pi(x,y)$ be an $n\times 1$ vector valued function on $\Omega\times\Omega$.
We say that a pair $(\vec G(x,y),\vec\Pi(x,y))$ is a Green function for the Stokes system  $\eqref{dp}$ if it satisfies the following properties:
\begin{enumerate}[a)]
\item
$\vec G(\cdot,y)\in W^{1,1}_0(\Omega)^{n\times n}$ and $\vec G(\cdot,y)\in W^{1,2}(\Omega\setminus B_R(y))^{n\times n}$ for all $y\in \Omega$ and $R>0$.
Moreover, $\vec \Pi(\cdot,y)\in L^1_0(\Omega)^n$ for all $y\in \Omega$.
\item
For any $y\in \Omega$, $(\vec G(\cdot,y),\vec \Pi(\cdot,y))$ satisfies
\begin{equation*}
\dv \vec G(\cdot,y)=0 \quad \text{in }\, \Omega
\end{equation*} 
and 
\begin{equation*}		
\sL\vec G(\cdot,y)+D \vec \Pi(\cdot,y)=\delta_y\vec I \quad \text{in }\, \Omega
\end{equation*}
in the sense that for any $1\le k\le n$ and $\vec \varphi\in C^\infty_0(\Omega)^n$, we have 
\[
\int_\Omega a_{\alpha\beta}^{ij}D_\beta G^{jk}(x,y)D_\alpha \varphi^i(x)\,dx-\int_\Omega \Pi^k(x,y)\dv \vec \varphi(x)\,dx=\varphi^k(y).
\] 
\item
If $(\vec u,p)\in W^{1,2}_0(\Omega)^n\times L^2_0(\Omega)$ is the weak solution of \eqref{dps} with  $\vec f,\, \vec f_\alpha\in L^\infty(\Omega)^n$ and $g\in {L}^\infty_0(\Omega)$, then we have 
\[
\vec u(x)=\int_\Omega \vec G(y,x)^{\tr}\vec f(y)\,dy-\int_\Omega D_\alpha \vec  G(y,x)^{\tr}\vec f_\alpha(y)\,dy-\int_\Omega \vec\Pi(x,y)g(y)\,dy.
\]
\end{enumerate}
\end{definition}
\begin{remark}
The $L^2$-solvability of the Stokes system with the Dirichlet boundary condition  (see Section \ref{1006@sec2}) and  the part c) of the above definition  give the uniqueness of a Green function.
Indeed, if $(\tilde{\vec G}(x,y), \tilde{\vec\Pi}(x,y))$ is another Green function for $\eqref{dp}$, then by the uniqueness of the solution, we have 
\[
\int_\Omega \vec G(y,x)^{\tr}\vec f(y)\,dy-\int_\Omega \vec \Pi(x,y)g(y)\,dy=\int_\Omega \tilde{\vec G}(y,x)^{\tr}\vec f(y)\,dy-\int_\Omega \tilde{\vec \Pi}(x,y)g(y)\,dy
\]
for any $\vec f\in C^\infty_0(\Omega)^n$ and $g\in C^\infty_0(\Omega)$.
Therefore, we conclude that $(\vec G, \vec \Pi)=(\tilde{\vec G},\tilde{\vec \Pi})$ a.e. in $\Omega\times \Omega$.
\end{remark}

\subsection{Existence of the Green function}		\label{0110.sec1}
To construct the Green function, we impose the following  conditions. 

\begin{A0}
There exist positive constants $R_1$ and $K_1$ such that the following holds:
for any $x_0\in \partial \Omega$ and $0<r\le R_1$, there is a coordinate system depending on $x_0$ and $r$ such that in the new coordinate system, we have 
$$
\Omega_r(x_0)=\{x\in B_r(x_0):x_1>\psi(x')\},
$$
where $\psi:\bR^{n-1}\to \bR$ is a Lipschitz function with $\Lip(\psi)\le K_1$.
\end{A0}

\begin{A1}
There exist constants $\mu\in (0,1]$ and $A_1>0$ such that the following holds:
if $(\vec u,p)\in W^{1,2}(B_R(x_0))^n\times L^2(B_R(x_0))$ satisfies
\begin{equation}		\label{160907@eq2}		
\left\{
\begin{aligned}
\sL\vec u+Dp=0 \quad \text{in }\, B_R(x_0),\\
\dv \vec u=0 \quad \text{in }\, B_R(x_0),
\end{aligned}
\right.
\end{equation}
where $x_0\in \Omega$ and $R\in (0,d_{x_0}]$, then we have 
\begin{equation}		\label{160907@eq3}
[\vec u]_{C^{\mu}(B_{R/2}(x_0))}\le A_1R^{-\mu}\left(\fint_{B_R(x_0)}\abs{\vec u}^2\,dx\right)^{1/2},
\end{equation}
where $[\vec u]_{C^\mu(B_{R/2}(x_0))}$ denotes the usual H\"older seminorm.
The statement is valid, provided that $\sL$ is replaced by $\sL^*$.
\end{A1}

\begin{theorem}		\label{1226.thm1}
Let $\Omega$ be a  domain in $\bR^n$ with $\diam(\Omega)\le K_0$, where $n\ge 3$.
Assume conditions $(\bf{A0})$ and  $(\bf{A1})$.
Then there exist  Green functions $(\vec G(x,y), \vec \Pi(x,y))$ and $(\vec G^*(x,y),\vec \Pi^*(x,y))$ for $\eqref{dp}$ and $\eqref{dps}$, respectively, such that the following identity:
\begin{equation}		\label{1226.eq1a}
\vec G(x,y)=\vec G^*(y,x)^{\tr}, \quad \forall x,\, y\in \Omega, \quad x\neq y.
\end{equation}
Also, for any $x,\, y\in \Omega$ satisfying $0<\abs{x-y}<d_y/2$, we have 
\begin{equation*}		
\abs{\vec G(x,y)}\le C\abs{x-y}^{2-n}.
\end{equation*} 
Moreover, for any $y\in \Omega$ and $R\in (0, d_y]$, we obtain
\begin{enumerate}[i)]
\item
$\norm{\vec G(\cdot,y)}_{L^{2n/(n-2)}(\Omega\setminus B_R(y))}+\norm{D\vec G(\cdot,y)}_{L^2(\Omega\setminus B_R(y))}\le CR^{(2-n)/2}$.
\item
$\abs{\set{x\in \Omega:\abs{\vec G(x,y)}>t}}\le Ct^{-n/(n-2)}$ for all $t>d_y^{2-n}$.
\item
$\abs{\set{x\in \Omega:\abs{D_x\vec G(x,y)}>t}}\le Ct^{-n/(n-1)}$ for all $t>d^{1-n}_y$.
\item
$\norm{\vec G(\cdot,y)}_{L^q(B_R(y))}\le C_qR^{2-n+n/q}$, where $ q\in [1,n/(n-2))$.
\item
$\norm{D\vec G(\cdot,y)}_{L^q(B_R(y))}\le C_qR^{1-n+n/q}$, where $q\in [1,n/(n-1))$
\item
$\norm{\vec \Pi(\cdot,y)}_{L^q(\Omega)}\le C_{y,q}$, where $q\in[1,n/(n-1))$.
\end{enumerate}
In the above, 
$C=C(n,\lambda,K_0,K_1,R_1,\mu,A_1)$,  $C_q=C_q(n,\lambda,K_0,K_1,R_1,\mu, A_1,q)$, and $C_{y,q}=C_{y,q}(n,\lambda,K_0,K_1,R_1,\mu,A_1,q,d_y)$.
The same estimates are also valid for $(\vec G^*(x,y), \vec \Pi^*(x,y))$.
\end{theorem}

\begin{remark}		
Let $(\vec u,p)\in W^{1,2}_0(\Omega)^n \times L^2_0(\Omega)$ be the weak solution of the problem
\[
\left\{
\begin{aligned}
\sL \vec u+D p=\vec f+D_\alpha\vec f_\alpha &\quad \text{in }\, \Omega,\\
\dv \vec u=0 &\quad \text{in }\, \Omega.
\end{aligned}
\right.
\]
Then by the property c) of Definition \ref{0110.def} and the identity \eqref{1226.eq1a}, we have the following representation for $\vec u$: 
\begin{equation*}	
\vec u(x):=\int_\Omega \vec G(x,y)\vec f(y)\,dy-\int_\Omega D_\alpha \vec G(x,y)\vec f_\alpha(y)\,dy.
\end{equation*}
Also, the following estimates are easy consequences of the identity \eqref{1226.eq1a} and the estimates i) -- v) in Theorem \ref{1226.thm1} for $\vec G^*(\cdot,x)$:
\begin{enumerate}[a)]
\item
$\norm{\vec G(x,\cdot)}_{L^{2n/(n-2)}(\Omega\setminus B_R(x))}+\norm{D\vec G(x,\cdot)}_{L^2(\Omega\setminus B_R(x))}\le CR^{(2-n)/2}$.
\item
$\abs{\set{y\in \Omega:\abs{\vec G(x,y)}>t}}\le Ct^{-n/(n-2)}$ for all $t>d_x^{2-n}$.
\item
$\abs{\set{y\in \Omega:\abs{D_y\vec G(x,y)}>t}}\le Ct^{-n/(n-1)}$ for all $t>d^{1-n}_x$.
\item
$\norm{\vec G(x,\cdot)}_{L^q(B_R(x))}\le C_qR^{2-n+n/q}$, where $ q\in [1,n/(n-2))$.
\item
$\norm{D\vec G(x,\cdot)}_{L^q(B_R(x))}\le C_qR^{1-n+n/q}$, where $q\in [1,n/(n-1))$.
\end{enumerate}
\end{remark}

In the theorem and the remark below, we show that  if the coefficients have a vanishing mean oscillation $(\VMO)$, then the condition $(\bf{A1})$ holds.

\begin{theorem}		\label{0110.thm1}
Suppose that the coefficients of $\sL$ belong to the class of $\VMO$; i.e. we have 
\[
\lim_{\rho\to 0}\omega_\rho(A_{\alpha\beta}):=\lim_{\rho\to 0}\sup_{x\in \bR^n}\sup_{s\le \rho}\fint_{B_s(x)}\bigabs{A_{\alpha\beta}-(A_{\alpha\beta})_{B_s(x)}}=0.
\]
If $(\vec u, p)\in W^{1,2}(B_R(x_0))^n\times L^2(B_R(x_0))$ satisfies \eqref{160907@eq2} with $x_0\in \Omega$ and $0<R\le \min\{d_{x_0},1\}$, then for any $\mu\in (0,1)$, the estimate \eqref{160907@eq3} holds with the constant $A_1$ depending only on $n$, $\lambda$, $\mu$, and the $\VMO$ modulus of the coefficients.
\end{theorem}

\begin{remark}		\label{160907@rem1}
In the above theorem, the constant $\min\{d_{x_0},1\}$ is interchangeable with $\min\{d_{x_0},c\}$ for any fixed $c\in (0, \infty)$, possibly at the cost of increasing the constant $A_1$. 
Setting $c=\diam\Omega$,  the condition $(\bf{A1})$ holds with the constant $A_1$ depending on $n$, $\lambda$, $\diam\Omega$, $\mu$, and the $\VMO$ modulus $\omega_\rho$ of coefficients.

\end{remark}

The following corollary is  immediate consequence of Theorem 
\ref{1226.thm1} and  Remark \ref{160907@rem1}.

\begin{corollary}		
Let $\Omega$ be a Lipschitz domain in $\bR^n$, where $n\ge 3$.
Suppose the coefficients of $\sL$ belong to the class of $\VMO$.
Then there exists the Green function for \eqref{dp} and it satisfies the  assertions in Theorem \ref{1226.thm1}.
\end{corollary}

\subsection{Global estimate of the Green function}		\label{0110.sec2}
We impose the following assumption to obtain the global pointwise estimate for the Green function.

\begin{A2}
There exists a constant $A_2>0$ such that if $(\vec u, p)\in W^{1,2}_0(\Omega)^n\times L^2_0(\Omega)$ satisfies
\begin{equation}		\label{170418@eq1}		
\left\{
\begin{aligned}
\sL\vec u+Dp=\vec f \quad \text{in }\, \Omega,\\
\dv \vec u=0 \quad \text{in }\,\Omega,
\end{aligned}
\right.
\end{equation}
where $\vec f\in L^\infty(\Omega)^n$, then $\vec u\in L^\infty(\Omega)^n$ with the estimate
$$
\|\vec u\|_{L^\infty(\Omega_{R/2}(x_0))}\le A_2\left(R^{-n/2}\|\vec u\|_{L^2(\Omega_R(x_0))}+R^2\|\vec f\|_{L^\infty(\Omega_R(x_0))}\right)
$$
for any $x_0\in \Omega$ and $0<R<\diam\Omega$.
The statement is valid, provided that $\sL$ is replaced by $\sL^*$.
\end{A2}

\begin{theorem}		\label{1226.thm2}
Let $\Omega$ be a  domain in $\bR^n$ with $\diam(\Omega)\le K_0$, where $n\ge 3$.
Assume conditions $(\bf{A0})$, $(\bf{A1})$, and $(\bf{A2})$.
Let $(\vec G(x,y),\vec \Pi(x,y))$ be the Green function for $\eqref{dp}$ in $\Omega$ as constructed in Theorem \ref{1226.thm1}.
Then we have  the  global pointwise estimate for $\vec G(x,y)$:
\begin{equation}		\label{0109.eq1}
\abs{\vec G(x,y)}\le C\abs{x-y}^{2-n}, \quad \forall x,\,y\in \Omega, \quad x\neq y,
\end{equation}
where $C=C(n,\lambda,K_0,K_1, R_1,A_2)$.
\end{theorem}

From the global $L^q$-estimates for the Stokes systems in Section \ref{sec_es}, we obtain an example of the condition $(\bf{A2})$ in the theorem below.
The proof of the theorem follows a standard localization argument; see Section \ref{0304.sec1} for the details.
Similar results for elliptic systems are given for the Dirichlet problem in \cite{MR2718661} and for the Neumann problem in \cite{MR3105752}.

\begin{theorem}		\label{0110.thm2}
Let $\Omega$ be a domain  in $\bR^n$ with $\diam(\Omega)\le K_0$, where $n\ge 3$.
Assume the condition $(\bf{A0})$ with a sufficiently small $K_1$, depending only on $n$ and $\lambda$.
If the coefficients of $\sL$ belong to the class of $\VMO$,
then the condition $(\bf{A2})$ holds with the constant $A_2$ depending only on $n$, $\lambda$, $K_0$, $R_1$, and the $\VMO$ modulus of the coefficients.
\end{theorem}

By combining Theorems \ref{1226.thm2} and \ref{0110.thm2}, we immediately obtain the following result.

\begin{corollary}		
Let $\Omega$ be a bounded $C^1$ domain  in $\bR^n$, where $n\ge 3$.
Suppose that the coefficients of $\sL$ belong to the class of $\VMO$.
Then there exists the Green function for $\eqref{dp}$ and it satisfies the global pointwise estimate \eqref{0109.eq1}.
\end{corollary}

\section{Some auxiliary results}

\subsection{$L^2$-solvability}		\label{1006@sec2}
In this subsection, we consider the existence theorem for weak solutions of the Stokes system with measurable coefficients.
For the solvability of the Stokes system, we impose the following condition.

\begin{D}
Let $\Omega$ be a bounded domain in $\bR^n$, where $n\ge 2$.
There exist a  linear operator $B:L^2_0(\Omega)\to W^{1,2}_0(\Omega)^n$ and a constant $A>0$ such that 
\[
\dv Bg=g\, \text{ in }\, \Omega \quad \text{and} \quad \norm{Bg}_{W^{1,2}_0(\Omega)}\le A\norm{g}_{L^2(\Omega)}.
\]
\end{D}

\begin{remark}		\label{K0122.rmk2}
It is well known that if $\Omega$ is a Lipschitz domain  with $\diam(\Omega)\le K_0$, which satisfies the condition $(\bf{A0})$, then for any $1<q<\infty$, there exists a bounded linear operator $B_q:L^q_0(\Omega)\to W^{1,q}_0(\Omega)^n$ such that 
\[
\dv B_q g=g \,\text{ in }\,\Omega, \quad \norm{D(B_q g)}_{L^q(\Omega)}\le C\norm{g}_{L^q(\Omega)},
\]
where the constant $C$ depends only on $n$, $q$,  $K_0$, $K_1$, and $R_1$; see e.g.,  \cite{MR2263708}.
We point out that if $\Omega=B_R(x)$ or $\Omega=B^+_R(x)$, then   
\begin{equation}		\label{0123.eq1a}
\norm{D(B_q g)}_{L^q(\Omega)}\le C\norm{g}_{L^q(\Omega)},
\end{equation}
where $C=C(n,q)$.
\end{remark}

\begin{lemma}		\label{122.lem1}
Assume the condition $(\bf{D})$.
Let
$$
q=\frac{2n}{n+2} \quad \text{if }\, n\ge3 \quad \text{and} \quad q=2 \quad \text{if }\,n=2.
$$
For $\vec f\in L^q(\Omega)^n$, $\vec f_\alpha\in L^2(\Omega)^n$, and $g\in L^2_0(\Omega)$, there exists a unique solution $(\vec u,p)\in W^{1,2}_0(\Omega)^n\times L^2_0(\Omega)$ of the problem 
\begin{equation}		\label{0204.eq2}		
\left\{
\begin{aligned}
\sL \vec u+D p=\vec f+D_\alpha\vec f_\alpha &\quad \text{in }\, \Omega,\\
\dv \vec u=g &\quad \text{in }\, \Omega.
\end{aligned}
\right.
\end{equation}
Moreover, we have  
\begin{equation}		\label{1227.eq1}
\norm{p}_{L^2(\Omega)}+\norm{D\vec u}_{L^2(\Omega)}\le C\left(\norm{\vec f}_{L^{q}(\Omega)}+\norm{\vec f_\alpha}_{L^2(\Omega)}+\norm{g}_{L^2(\Omega)}\right),
\end{equation}
where $C=C(n,\lambda,A)$ if $n\ge 3$ and $C=C(\lambda,A,|\Omega|)$ if $n=2$.
In the case when  $\Omega=B_R(x)$ or $\Omega=B^+_R(x)$, if $\vec f\in L^2(\Omega)^n$, then we have
\begin{equation}		\label{122.eq1a}
\norm{p}_{L^2(\Omega)}+\norm{D\vec u}_{L^2(\Omega)}\le C'\left(R\norm{\vec f}_{L^{2}(\Omega)}+\norm{\vec f_\alpha}_{L^2(\Omega)}+\norm{g}_{L^2(\Omega)}\right),
\end{equation}
where $C'=C'(n,\lambda)$.
\end{lemma}

\begin{proof} 
We mainly follow the argument given by Maz'ya-Rossmann \cite[Theorem 5.2]{MR2321139}.
Also see \cite[Theorem 3.1]{MR3320459}.
Let $H(\Omega)$ be the Hilbert space consisting of functions $\vec u\in W^{1,2}_0(\Omega)^n$ such that $\dv \vec u=0$ and $H^\bot(\Omega)$ be orthogonal complement of $H(\Omega)$ in $W^{1,2}_0(\Omega)^n$.
We also define $P$ as the orthogonal projection from $W^{1,2}_0(\Omega)^n$ onto $H^\bot(\Omega)$.
Then, one can easily show that the operator $\cB=P\circ B:L^2_0(\Omega)\to H^\bot(\Omega)$ is bijective.
Moreover, we obtain for $g\in L^2_0(\Omega)$ that 
\begin{equation}		\label{0123.eq1}
\dv \cB g=g \,\text{ in }\,\Omega, \quad \norm{\cB g}_{W^{1,2}(\Omega)}\le A\norm{g}_{L^2(\Omega)}.
\end{equation}

Now, let $\vec f,\, \vec f_\alpha\in L^2(\Omega)^n$ and $g\in L^2_0(\Omega)$.
Then from the above argument, there exists a unique function $\vec w:=\cB g\in H^\bot(\Omega)$ such that  \eqref{0123.eq1}  is satisfied.
Also, by the Lax-Milgram theorem, one can find the function  $\vec v\in H(\Omega)$ that satisfies 
\begin{equation*}		
\int_\Omega A_{\alpha\beta}D_\beta \vec v\cdot D_\alpha \vec \varphi\,dx=\int_\Omega \vec f\cdot \vec \varphi\,dx-\int_\Omega \vec f_\alpha \cdot D_\alpha \vec \varphi\,dx-\int_\Omega A_{\alpha\beta}D_\beta \vec w\cdot D_\alpha \vec \varphi\,dx
\end{equation*}
for all $\vec \varphi\in H(\Omega)$.
By setting $\vec \varphi=\vec v$ in the above identity, and then, using H\"older's inequality and the Sobolev inequality, we have  
\[
\norm{D\vec v}_{L^2(\Omega)}\le C\left(\norm{\vec f}_{L^q(\Omega)}+\norm{\vec f_\alpha}_{L^2(\Omega)}+\norm{D\vec w}_{L^2(\Omega)}\right),
\]
where $q=2$ if $n=2$ and $q=2n/(n+2)$ if $n\ge 3$.
Therefore, the function $\vec u=\vec v+\vec w$ satisfies $\dv \vec u= g$ in $\Omega$ and the following identity:
\begin{equation}		\label{1227.eq1b}
\int_\Omega A_{\alpha\beta}D_\beta \vec u\cdot D_\alpha \vec \varphi\,dx=\int_\Omega \vec f\cdot \vec \varphi\,dx-\int_\Omega \vec f_\alpha \cdot D_\alpha \vec \varphi\,dx, \quad \forall \vec \varphi\in H(\Omega).
\end{equation}
Moreover, we have 
\begin{equation}		\label{1227.eq2a}
\norm{D\vec u}_{L^2(\Omega)}\le C\left(\norm{\vec f}_{L^{q}(\Omega)}+\norm{\vec f_\alpha}_{L^2(\Omega)}+\norm{g}_{L^2(\Omega)}\right).
\end{equation}
To find $p$, we let 
\[
\ell(\phi)=\int_\Omega A_{\alpha\beta}D_\beta \vec u\cdot D_\alpha (\cB\tilde{\phi})\,dx-\int_\Omega \vec f\cdot \cB\tilde{\phi}\,dx+\int_\Omega \vec f_\alpha \cdot D_\alpha(\cB\tilde{\phi})\,dx,
\]
where $\phi\in L^2(\Omega)$ and $\tilde{\phi}=\phi-(\phi)_\Omega\in L^2_0(\Omega)$.
Since 
\[
\norm{\cB\tilde{\phi}}_{W^{1,2}(\Omega)}\le A\norm{\tilde{\phi}}_{L^2(\Omega)}\le C(n,A)\norm{\phi}_{L^2(\Omega)},
\]
$\ell$ is a bounded linear functional on $L^2(\Omega)$.
Therefore, there exists a function $p_0\in L^2(\Omega)$ so that 
\[
\int_\Omega p_0 \tilde{\phi}\,dx=\ell (\tilde\phi), \quad \forall \tilde{\phi}\in L^2_0(\Omega),
\]
and thus, $p=p_0-(p_0)_\Omega\in L^2_0(\Omega)$ also satisfies the above identity.
Then by using the fact that $\cB(L^2_0(\Omega))=H^\bot(\Omega)$, we  obtain 
\begin{equation}		\label{122.eq1}
\int_\Omega A_{\alpha\beta}D_\beta \vec u\cdot D_\alpha \vec \varphi\,dx-\int_\Omega p \dv \vec \varphi\,dx=\int_\Omega \vec f \cdot \vec \varphi\,dx-\int_\Omega \vec f_\alpha \cdot D_\alpha \vec \varphi\,dx
\end{equation}
for all $\vec \varphi\in H^\bot(\Omega)$.
From \eqref{1227.eq1b} and \eqref{122.eq1}, we find that $(\vec u,p)$ is the weak solution of the problem \eqref{0204.eq2}.
Moreover, by setting  $\vec \varphi=\cB p$ in \eqref{122.eq1}, we have 
\begin{equation*}
\norm{p}_{L^2(\Omega)}\le C\left(\norm{D\vec u}_{L^2(\Omega)}+\norm{\vec f}_{L^{q}(\Omega)}+\norm{\vec f_\alpha}_{L^2(\Omega)}\right),
\end{equation*}
and thus, we get \eqref{1227.eq1} from \eqref{1227.eq2a}.

To establish \eqref{122.eq1a}, we observe that  
\[
\norm{u}_{L^2(\Omega)}\le C(n)R\norm{Du}_{L^2(\Omega)}, \quad \forall u\in W^{1,2}_0(\Omega), 
\]
provided that $\Omega=B_R(x)$ or $\Omega=B_R^+(x)$.
By using the above inequality and \eqref{0123.eq1a}, and following the same argument as above, one can easily show that the estimate \eqref{122.eq1a} holds.
The lemma is proved.
\end{proof}

\subsection{Interior estimates}
In this subsection we derive some interior estimates of $\vec u$ and $p$.
We start with the following Caccioppoli type inequality that can be found, for instance, in \cite{arXiv:1604.02690v2,MR0641818}.

\begin{lemma}		\label{1006@lem1}
Assume that $(\vec u,p)\in W^{1,2}(B_R(x_0))^n\times L^2(B_R(x_0))$ satisfies 
$$
\left\{
\begin{aligned}
\sL\vec u+Dp=0 &\quad \text{in }\, B_R(x_0),\\
\dv \vec u=0 &\quad \text{in }\, B_R(x_0),
\end{aligned}
\right.
$$
where $x_0\in \bR^n$ and $R>0$.
Then we have 
$$
\int_{B_{R/2}(x_0)}\bigabs{p-(p)_{B_{R/2}(x_0)}}^2\,dx+\int_{B_{R/2}(x_0)}\abs{D\vec u}^2\,dx\le CR^{-2}\int_{B_R(x_0)}\abs{\vec u}^2\,dx,
$$
where $C=C(n,\lambda)$.
\end{lemma}

\begin{proof}
Let $r\in (0, R]$ and denote $B_r=B_r(x_0)$.
By Remark \ref{K0122.rmk2},  there exists $\vec \phi\in W^{1,2}_0(B_r)^n$ such that 
\begin{equation*}		
\dv \vec \phi=p-(p)_{B_{r}} \,\text{ in }\, B_{r}
\end{equation*}
and 
\begin{equation*}
\norm{\vec\phi}_{L^{2n/(n-2)}(B_{r})}\le C\norm{D\vec \phi}_{L^2(B_{r})}\le C\norm{p-(p)_{B_{r}}}_{L^2(B_{r})},
\end{equation*}
where $C=C(n)$.
Since 
\begin{equation}		\label{160831@eq4}
\sL\vec u+D(p-(p)_{B_{r}})=0 \quad \text{in }\, B_r,
\end{equation}
by testing with $\vec \phi$ in \eqref{160831@eq4}, we get  
\begin{equation}		\label{1006@eq1b}
\int_{B_{r}}\abs{p-(p)_{B_{r}}}^2\,dx\le C_1\int_{B_{r}}\abs{D\vec u}^2\,dx, \quad \forall r\in (0,R],
\end{equation}
where $C_1=C_1(n,\lambda)$.
From the above inequality, it remains us to show that 
\begin{equation}		\label{160831@eq4a}
\int_{B_{R/2}}|D\vec u|^2\,dx\le CR^{-2}\int_{B_R}|\vec u|^2\,dx.
\end{equation}
Let $0<\rho_1<\rho_2\le R$ and $\delta\in (0,1)$.
Let $\eta$ be a smooth function on $\bR^d$ such that 
$$
0\le \eta\le 1, \quad \eta=1 \quad \text{in }\, B_{\rho_1}, \quad \supp \eta\subset B_{\rho_2}, \quad |D\eta|\le C(d)(\rho_2-\rho_1)^{-1}.
$$
Then by applying $\eta^2 \vec u$ as a test function to 
$$
\sL\vec u+D(p-(p)_{B_{\rho_2}})=0 \quad \text{in }\, B_R
$$
 and using the fact that $\dv u=0$,  we have 
$$
\int_{B_R}A_{\alpha\beta}\eta D_\beta \vec u\cdot \eta D_\alpha \vec u\,dx=-2\int_{B_R}A_{\alpha\beta}\eta D_\beta \vec u\cdot D_\alpha \eta \vec u\,dx+2\int_{B_R}(p-(p)_{B_{\rho_2}})\eta D\eta \cdot \vec u\,dx,
$$
and thus, by the ellipticity condition, H\"older's inequality, and Young's inequality, we obtain
$$
\int_{B_{\rho_1}}|D\vec u|^2\,dx\le \frac{C_\delta}{(\rho_2-\rho_1)^2}\int_{B_{\rho_2}}|\vec u|^2\,dx+\frac{\delta}{C_1}\int_{B_{\rho_2}}|p-(p)_{B_{\rho_2}}|^2\,dx,
$$
where $C_\delta=C_\delta(n,\lambda,\delta)$, and $C_1$ is the constant in \eqref{1006@eq1b}.
From this together with \eqref{1006@eq1b}, it follows that 
\begin{equation}		\label{160831@eq5a}
\int_{B_{\rho_1}}|D\vec u|^2\,dx\le \frac{C_\delta}{(\rho_2-\rho_1)^2}\int_{B_{\rho_2}}|\vec u|^2\,dx+\delta\int_{B_{\rho_2}}|D\vec u|^2\,dx.
\end{equation}
Let us set
$$
\delta=\frac{1}{8}, \quad \rho_k=\frac{R}{2}\left(2-\frac{1}{2^k}\right), \quad k=0,1,2,\ldots.
$$
Then by \eqref{160831@eq5a}, we have 
$$
\int_{B_{\rho_k}}|D\vec u|^2\,dx\le \frac{C4^k}{R^2}\int_{B_{\rho_{k+1}}}|\vec u|^2\,dx+\delta\int_{B_{\rho_{k+1}}}|D\vec u|^2\,dx, \quad k\in \{0,1,2,\ldots\},
$$
where $C=C(n,\lambda)$.
By multiplying both sides of the above inequality by $\delta^k$ and summing the terms with respect to $k=0,1,\ldots$, we obtain 
$$
\sum_{k=0}^\infty \delta^k\int_{B_{\rho_k}}|D\vec u|^2\,dx\le \frac{C}{R^2}\sum_{k=0}^\infty (4\delta)^k\int_{B_{\rho_{k+1}}}|\vec u|^2\,dx+\sum_{k=1}^\infty \delta^k\int_{B_{\rho_{k}}}|D\vec u|^2\,dx.
$$
By subtracting the last term of the right-hand side in the above inequality, we obtain the desired estimate \eqref{160831@eq4a}.
The lemma is proved.
\end{proof}

\begin{lemma}		\label{1227.lem1}
Assume the condition $(\bf{A1})$.
Let $(\vec u,p)\in W^{1,2}(B_R(x_0))^n\times L^2(B_R(x_0))$ satisfy
\begin{equation}		\label{0302.eq1}
\left\{
\begin{aligned}
\sL\vec u+Dp=0 \quad \text{in }\, B_R(x_0),\\
\dv \vec u=0 \quad \text{in }\, B_R(x_0),
\end{aligned}
\right.
\end{equation}
where $x_0\in \Omega$ and $R\in(0,d_{x_0}]$.
Then we have 
\begin{equation}		\label{0929.eq3}
\int_{B_r(x_0)}\abs{D\vec u}^2\,dx\le C_1\left(\frac{r}{s}\right)^{n-2+2\mu}\int_{B_s(x_0)}\abs{D\vec u}^2\,dx, \quad 0<r<s\le R,
\end{equation}
where $C_1=C_1(n,\lambda,A_1)$.
Moreover, we get 
\begin{equation}		\label{0929=e1}
\norm{\vec u}_{L^\infty(B_{R/2}(x_0))}\le C_2 R^{-n}\norm{\vec u}_{L^1(B_R(x_0))},
\end{equation}
where $C_2=C_2(n,\mu, A_1)$.
The statement is valid, provided that $\sL$ is replaced by $\sL^*$.
\end{lemma}

\begin{proof}
To prove \eqref{0929.eq3}, we only need to consider the case $0<r\le s/4$.
Also, by replacing $\vec u-(\vec u)_{B_s(x_0)}$ if necessary, we may assume that $(\vec u)_{B_s(x_0)}=0$.
Since $(\vec u-(\vec u)_{B_{2r}(x_0)},p)$ is a weak solution of \eqref{0302.eq1}, we get from Lemma \ref{1006@lem1} that 
\begin{equation*}		
\int_{B_{r}(x_0)}\abs{D\vec u}^2\,dx\le  Cr^{-2}\int_{B_{2r}(x_0)}\abs{\vec u-(\vec u)_{B_{2r}(x_0)}}^2\,dx.
\end{equation*}
By  $(\bf{A1})$, the Poincar\'e inequality, and the above inequality, we have  
\begin{align*}
\int_{B_r(x_0)}\abs{D\vec u}^2\,dx&\le Cr^{n-2+2\mu}[\vec u]^2_{C^{\mu}(B_{2r}(x_0))}\le Cr^{n-2+2\mu}[\vec u]^2_{C^{\mu}(B_{s/2}(x_0))}\\
&\le CA_1^2 r^{n-2+2\mu}s^{-n-2\mu}\int_{B_s(x_0)}\abs{\vec u}^2\,dx\le CA_1^2 \left(\frac{r}{s}\right)^{n-2+2\mu}\int_{B_s(x_0)}\abs{D\vec u}^2\,dx,
\end{align*}
which establishes \eqref{0929.eq3}.

We observe that $(\bf{A1})$ and a well known averaging argument yield
\begin{equation}		\label{0316.eq1}
\norm{\vec u}_{L^\infty(B_{R/2}(x_0))}\le C\left(\fint_{B_R(x_0)}\abs{\vec u}^2\,dx\right)^{1/2},
\end{equation}
for any $R\in (0,d_{x_0}]$, where $C=C(n, \mu, A_1)$.
For the proof that \eqref{0316.eq1} implies \eqref{0929=e1}, we refer to \cite[pp. 80-82]{MR1239172}.
\end{proof}

\begin{lemma}		
Let $\Omega$ be a  domain in $\bR^n$ with $\diam(\Omega)\le K_0$, where $n\ge 3$.
Assume  conditions $(\bf{A0})$ and  $(\bf{A1})$.
Let $(\vec u,p)\in W^{1,2}_0(\Omega)^n \times L^2_0(\Omega)$ be a solution of the problem 
\begin{equation*}		
\left\{
\begin{aligned}
\sL \vec u+D p=\vec f &\quad \text{in }\, \Omega,\\
\dv \vec u=0 &\quad \text{in }\, \Omega,
\end{aligned}
\right.
\end{equation*}
where $\vec f\in L^\infty(\Omega)^n$. 
Then for any $x_0\in \Omega$ and $R\in (0,d_{x_0}]$, $\vec u$ is continuous in $B_R(x_0)$ with the estimate
\begin{equation}		\label{1229.eq1}
[\vec u]_{C^{\mu_1}(B_{R/2}(x_0))}\le C\left(R^{-n/2+1-\mu_1}\norm{D\vec u}_{L^2(\Omega)}+\norm{\vec f}_{L^q(\Omega)}\right)
\end{equation}
for any $q\in\big(\frac{n}{2},\frac{n}{2-\mu}\big)$, where $\mu_1:=2-n/q$ and  $C=C(n,\lambda,\mu,A_1,q)$.
Moreover, if $\vec f$ is supported in $B_R(x_0)$, then we have 
\begin{equation}		\label{1229.eq1a}
\norm{\vec u}_{L^\infty(B_{R/2}(x_0))}\le CR^2\norm{\vec f}_{L^\infty(B_R(x_0))},
\end{equation}
where $C=C(n,\lambda,K_0, K_1,R_1,\mu, A_1)$.
The statement is valid, provided that $\sL$ is replaced by $\sL^*$.
\end{lemma}

\begin{proof}
Let $x\in B_{R/2}(x_0)$ and $0<s\le R/2$.
We decompose $(\vec u,p)$ as $(\vec u_1,p_1)+(\vec u_2,p_2)$, where $(\vec u_2,p_2)\in W^{1,2}_0(B_s(x))^n\times {L}^2_0(B_s(x))$ satisfies 
\[
\left\{
\begin{aligned}
\sL \vec u_2+D p_2=\vec f &\quad \text{in }\, B_s(x),\\
\dv \vec u_2=0 &\quad \text{in }\, B_s(x).
\end{aligned}
\right.
\]
And then $(\vec u_1, p_1)\in W^{1,2}(B_s(x))^n\times L^2(B_s(x))$ satisfies
\[
\left\{
\begin{aligned}
\sL \vec u_1+D p_1=0 &\quad \text{in }\, B_s(x),\\
\dv \vec u_1=0 &\quad \text{in }\, B_s(x).
\end{aligned}
\right.
\]
From the estimate \eqref{1227.eq1} and H\"older's inequality, it follows that  
\begin{equation}		\label{0929.eq5a}
\norm{D\vec u_2}_{L^2(B_s(x))}\le C\norm{\vec f}_{L^{2n/(n+2)}(B_s(x))}\le Cs^{n/2-1+\mu_1}\norm{\vec f}_{L^{q}(B_R(x_0))},
\end{equation}
where $q\in \big(\frac{n}{2}, \frac{n}{2-\mu}\big)$, $\mu_1=2-n/q$, and $C=C(n,\lambda, q)$.
For $0<r<s$, we obtain by Lemma \ref{1227.lem1} that 
\begin{align}
\nonumber
\int_{B_r(x)}\abs{D\vec u}^2\,dx&\le 2\int_{B_r(x)}\abs{D\vec u_1}^2\,dx+2\int_{B_r(x)}\abs{D\vec u_2}^2\,dx\\
\nonumber
&\le C\left(\frac{r}{s}\right)^{n-2+2\mu}\int_{B_s(x)}\abs{D\vec u_1}^2\,dx+2\int_{B_s(x)}\abs{D\vec u_2}^2\,dx\\
\label{0929.eq5}
&\le C\left(\frac{r}{s}\right)^{n-2+2\mu}\int_{B_s(x)}\abs{D\vec u}^2\,dx+C\int_{B_s(x)}\abs{D\vec u_2}^2\,dx,
\end{align}
where $C=C(n,\lambda,A_1)$.
Therefore we get from \eqref{0929.eq5a} and \eqref{0929.eq5} that 
\begin{equation*}		
\int_{B_r(x)}\abs{D\vec u}^2\,dx\le C\left(\frac{r}{s}\right)^{n-2+2\mu}\int_{B_s(x)}\abs{D\vec u}^2\,dx+Cs^{n-2+2\mu_1}\norm{\vec f}^2_{L^{q}(B_R(x_0))}.
\end{equation*}
Then by \cite[Lemma 2.1, p. 86]{MR0717034}, we have 
\begin{equation*}	
\int_{B_r(x)}\abs{D\vec u}^2\,dx\le C\left(\frac{r}{R}\right)^{n-2+2\mu_1}\int_\Omega \abs{D\vec u}^2\,dx+Cr^{n-2+2\mu_1}\norm{\vec f}_{L^{q}(B_R(x_0))}^2
\end{equation*}
for any $x\in B_{R/2}(x_0)$ and $r\in (0,R/2)$.
From this together with  Morrey-Campanato's theorem, we prove \eqref{1229.eq1}.

To see \eqref{1229.eq1a}, assume $\vec f$ is supported in $B_R(x_0)$.
Notice from the Sobolev inequality that 
\[
\norm{\vec u}_{L^2(B_R(x_0))}\le C(n)R\norm{D\vec u}_{L^2(\Omega)}.
\]
Then we obtain by \eqref{1229.eq1} and the above estimate that 
\begin{align}
\nonumber
\norm{\vec u}_{L^\infty(B_{R/2}(x_0))}&\le CR^{\mu_1}[\vec u]_{C^{\mu_1}(B_{R/2}(x_0))}+CR^{n/2}\norm{\vec u}_{L^2(B_R(x_0))}\\
\nonumber
&\le CR^{-n/2+1}\norm{D\vec u}_{L^2(\Omega)}+CR^2\norm{\vec f}_{L^{\infty}(B_R(x_0))},
\end{align}
and thus, we get desired estimate from the inequality \eqref{1227.eq1}.
The lemma is proved.
\end{proof}

\section{Proofs of main theorems}		\label{1006@sec1}
In the section, we prove main theorems stated in Sections \ref{0110.sec1} and \ref{0110.sec2}.

\subsection{Proof of Theorem \ref{1226.thm1}}		\label{0204.sec1}

\subsubsection{Averaged Green function}		\label{0108.sec1}
Let $y\in \Omega$ and $\epsilon>0$ be fixed, but arbitrary.
Fix an integer $1\le k\le n$ and  let $(\vec v_\epsilon,\pi_\epsilon)=(\vec v_{\epsilon;y,k},\pi_{\epsilon;y,k})$ be the solution in $W^{1,2}_0(\Omega)^n\times {L}^2_0(\Omega)$ of 
\begin{equation*}		
\left\{
\begin{aligned}
\sL \vec u+D p=\frac{1}{\abs{\Omega_\epsilon(y)}}1_{\Omega_\epsilon(y)} \vec e_k &\quad \text{in }\, \Omega,\\
\dv \vec u=0 &\quad \text{in }\, \Omega,
\end{aligned}
\right.
\end{equation*}
where $\vec e_k$ is the $k$-th unit vector in $\bR^n$.
We define \emph{the averaged Green function} $(\vec G_\epsilon(\cdot,y), \vec \Pi_\epsilon(\cdot,y))$ for $\eqref{dp}$ by setting
\begin{equation}		\label{0102.eq1b}
G_\epsilon^{jk}(\cdot,y)=v_{\epsilon;y,k}^j \quad \text{and}\quad \Pi_\epsilon^k(\cdot,y)=\pi_{\epsilon;y,k}.
\end{equation}
Then $(\vec G_\epsilon(\cdot,y), \vec \Pi_\epsilon(\cdot,y))$ satisfies 
\begin{equation}		\label{0927.eq2a}
\int_\Omega a_{\alpha\beta}^{ij}D_\beta G^{jk}_\epsilon(\cdot,y)D_\alpha \varphi^i\,dx-\int_\Omega \Pi_\epsilon^k(\cdot,y)\dv \vec \varphi\,dx=\fint_{\Omega_{\epsilon}(y)} \varphi^k\,dx
\end{equation}
for any $\vec \varphi\in W^{1,2}_0(\Omega)^n$.
We also obtain by \eqref{1227.eq1} that 
\begin{equation}		\label{0927.eq2b}
\norm{\vec \Pi_\epsilon(\cdot,y)}_{L^2(\Omega)}+\norm{D\vec G_\epsilon(\cdot,y)}_{L^2(\Omega)}\le C\epsilon^{(2-n)/2}, \quad \forall \epsilon>0,
\end{equation}
where $C=C(n,\lambda,K_0, K_1,R_1)$.
The following lemma is an immediate consequence of Lemma \ref{1006@lem1}.

\begin{lemma}	\label{0102-lem1}
Let $y\in \Omega$ and $\epsilon>0$.
\begin{enumerate}[(i)]
\item
For any $x_0\in \Omega$ and $R\in (0,d_{x_0}]$ satisfying $B_{R}(x_0)\cap B_\epsilon(y)=\emptyset$, we have 
\begin{equation*}		
\int_{B_{R/2}(x_0)}\abs{D\vec G_\epsilon(x,y)}^2\,dx \le CR^{-2} \int_{B_{R}(x_0)}\abs{\vec G_\epsilon(x,y)}^2\,dx,
\end{equation*}
where $C=C(n,\lambda)$.
\item
Let $R\in (0,2d_y/3]$ and $\epsilon\in (0,R/4)$.
Then we have 
\begin{equation*}		
\int_{B_{R}(y)\setminus B_{R/2}(y)}\abs{D\vec G_\epsilon(x,y)}^2\,dx\le CR^{-2}\int_{B_{3R/2}(y)\setminus B_{R/4}(y)}\abs{\vec G_\epsilon(x,y)}^2\,dx,
\end{equation*}
where $C=C(n,\lambda)$.
\end{enumerate}
\end{lemma}

With the preparations in the previous section, we obtain the  pointwise estimate of the averaged Green function $\vec G_\epsilon(\cdot,y)$.

\begin{lemma}		\label{0929-thm1}
There exists a constant $C=C(n,\lambda,K_0, K_1,R_1,\mu, A_1)>0$ such that for any $x,\, y\in \Omega$ satisfying $0<\abs{x-y}<d_y/2$, we have 
\begin{equation}		\label{0929-e2}
\abs{\vec G_\epsilon(x,y)}\le C\abs{x-y}^{2-n}, \quad \forall \epsilon\in(0,\abs{x-y}/3).
\end{equation}
\end{lemma}

\begin{proof}
Let $y\in \Omega$, $R\in (0,d_y)$, and  $\epsilon\in(0,R/2)$.
We denote $\vec v_\epsilon$ to be the $k$-th column of $\vec G_\epsilon(\cdot,y)$.
Assume that  $(\vec u,p)\in W^{1,2}_0(\Omega)^n\times L^2_0(\Omega)$ is the solution of 
\begin{equation}		\label{1229.eq2}		
\left\{
\begin{aligned}
\sL^* \vec u+D p=\vec f &\quad \text{in }\, \Omega,\\
\dv \vec u=0 &\quad \text{in }\, \Omega,
\end{aligned}
\right.
\end{equation}
where 
$f^i(x)= 1_{B_R(y)}\sgn (v_\varepsilon^i(x))$ and $\vec f=(f^1,\ldots,f^n)\in L^\infty(\Omega)^n$.
Then by testing with $\vec v_\epsilon$ in \eqref{1229.eq2}, we have 
\begin{equation*}	
\int_\Omega A_{\alpha\beta}D_\beta \vec v_\epsilon\cdot D_\alpha \vec u\,dx=\int_{B_R(y)}  \vec f\cdot  \vec v_\epsilon\,dx.
\end{equation*} 
Similarly, we set  $\vec \varphi=\vec u$ in \eqref{0927.eq2a} to obtain 
\begin{equation*}		
\int_\Omega A_{\alpha\beta}D_\beta \vec v_\epsilon\cdot D_\alpha \vec u\,dx=\fint_{B_\epsilon(y)} u^k\,dx.
\end{equation*}
From the above  two identities, we get
\begin{equation}		\label{0105.eq0}		
\int_{B_R(y)}\vec f\cdot \vec v_\epsilon\,dx=\fint_{B_\epsilon(y)} u^k\,dx,
\end{equation}
and thus, by  \eqref{1229.eq1a}, we derive 
\begin{equation}		\label{1229.eq2a}
\norm{\vec G_\epsilon(\cdot,y)}_{L^1(B_R(y))}\le CR^2, \quad R\in (0,d_y), \quad \epsilon\in (0,R/2),
\end{equation}
where $C=C(n,\lambda,K_0, K_1,R_1,\mu,A_1)$.

Now, we are ready to prove the lemma.
Let $x,\, y\in \Omega$ satisfy $0<\abs{x-y}<d_y/2$.
We write $R:=2\abs{x-y}/3$.
Note that if $\epsilon <R/2$, then $(\vec G_\epsilon(\cdot,y),\vec \Pi_\epsilon(\cdot,y))$ satisfies 
\[
\left\{
\begin{aligned}
\sL\vec G_\epsilon(\cdot,y)+D\vec \Pi_\epsilon(\cdot,y)=0 &\quad \text{in }\, B_{R}(x),\\
\dv \vec G_\epsilon(\cdot,y)=0 &\quad \text{in }\, B_R(x).
\end{aligned}
\right.
\]
Then by Lemma \ref{1227.lem1}, we have 
\begin{equation*}
\abs{\vec G_\epsilon(x,y)}\le CR^{-n}\norm{\vec G_\epsilon(\cdot,y)}_{L^1(B_R(x))}\le CR^{-n}\norm{\vec G_\epsilon(\cdot,y)}_{L^1(B_{3R}(y))}.
\end{equation*}
This together with \eqref{1229.eq2a} yields \eqref{0929-e2}.
The lemma is proved.
\end{proof}

Based on the pointwise estimate \eqref{0929-e2}, we prove that 
$\vec G_\epsilon(\cdot,y)$ and  $\vec \Pi_\epsilon(\cdot,y)$ satisfy  the following $L^q$-estimates uniformly in $\epsilon>0$.

\begin{lemma}		\label{1229-lem1}
Let $y\in \Omega$, $R\in (0,d_y]$, and $\epsilon>0$.
Then we have 
\begin{equation}		\label{1229.eq2b}
\norm{\vec G_\epsilon(\cdot,y)}_{L^{2n/(n-2)}(\Omega\setminus B_R(y))}+\norm{D\vec G_\epsilon(\cdot,y)}_{L^2(\Omega\setminus B_R(y))}\le CR^{(2-n)/2}.
\end{equation}
Also,  we obtain
\begin{align}
\label{0103.eq3}
\abs{\set{x\in \Omega:\abs{\vec G_\epsilon(x,y)}>t}}&\le Ct^{-n/(n-2)}, \quad \forall t>d_y^{2-n},\\
\label{0103.eq3a}
\abs{\set{x\in \Omega:\abs{D_x\vec G_\epsilon(x,y)}>t}}&\le Ct^{-n/(n-1)}, \quad \forall t>d_y^{1-n}.
\end{align}
Moreover, we derive the following uniform $L^q$ estimates:
\begin{align}		
\label{1229.eq2c}
\norm{\vec G_\epsilon(\cdot,y)}_{L^q(B_R(y))}\le C_qR^{2-n+n/q}, &\quad q\in [1,n/(n-2)),\\
\label{1229.eq2d}
\norm{D\vec G_\epsilon(\cdot,y)}_{L^q(B_R(y))}\le C_qR^{1-n+n/q}, &\quad q\in [1,n/(n-1)),\\
\label{0103.eq6}
\norm{\vec \Pi_\epsilon(\cdot,y)}_{L^q(\Omega)}\le C_{y,q}, &\quad q\in [1,n/(n-1)).
\end{align}
In the above, $C=C(n,\lambda,K_0,K_1,R_1, \mu, A_1)$, $C_q=C_q(n,\lambda,K_0, K_1,R_1, \mu,A_1,q)$, and $C_{y,q}=C_{y,q}(n,\lambda,K_0, K_1,R_1, \mu, A_1,q,d_y)$.
\end{lemma}

\begin{proof}
Recall the notation \eqref{0102.eq1b}.
We first prove the estimate \eqref{1229.eq2b}.
From the obvious fact that $d_y/3$ and $d_y$ are comparable to each other, we only need to prove the estimate \eqref{1229.eq2b} for $R\in (0,d_y/3]$.
If $\epsilon\ge R/12$, then by \eqref{0927.eq2b} and the Sobolev inequality, we have  
\begin{equation}		\label{K1230.eq1}
\norm{\vec G_\epsilon(\cdot,y)}_{L^{2n/(n-2)}(\Omega\setminus B_R(y))}+\norm{D\vec G_\epsilon(\cdot,y)}_{L^2(\Omega\setminus B_R(y))}\le C\norm{D\vec G_\epsilon(\cdot,y)}_{L^2(\Omega)}\le CR^{(2-n)/2}.
\end{equation}
On the other hand, if $\epsilon\in (0,R/12)$, then by setting $\vec \varphi=\eta^2\vec v_\epsilon$ in \eqref{0927.eq2a}, where  $\eta$ is a smooth function satisfying 
\begin{equation*}		
0\le \eta\le1, \quad \eta\equiv 1 \,\text{ on }\, \bR^n\setminus B_{R}(y), \quad \eta\equiv 0 \,\text{ on }\,B_{R/2}(y), \quad \abs{D\eta}\le CR^{-1},
\end{equation*}
we have 
\begin{equation}		\label{K1229.eq3a}
\int_\Omega \eta^2\abs{D\vec v_\epsilon}^2\,dx\le C\int_\Omega \abs{D\eta}^2\abs{\vec v_\epsilon}^2\,dx+C\int_{D}\abs{\pi_\epsilon-(\pi_\epsilon)_D}^2\,dx,
\end{equation}
where $D=B_R(y)\setminus B_{R/2}(y)$.
By Remark \ref{K0122.rmk2},  there exists a  function $\vec \phi_\epsilon\in W^{1,2}_0(D)^n$ such that 
\begin{equation*}		
\dv \vec \phi_\epsilon=\pi_\epsilon-(\pi_\epsilon)_{D} \,\text{ in }\, D, \quad \norm{D\vec \phi_\epsilon}_{L^2(D)}\le C\norm{\pi_\epsilon-(\pi_\epsilon)_D}_{L^2(D)},
\end{equation*}
where $C=C(n)$.
Therefore, by setting $\vec \varphi=\vec \phi_\epsilon $ in \eqref{0927.eq2a}, we get from  Lemma \ref{0102-lem1} (ii) that  
\begin{equation}		\label{1229.eq3b}
\int_{D}\abs{\pi_\epsilon-(\pi_\epsilon)_D}^2\,dx\le C\int_{D}\abs{D\vec v_\epsilon}^2\,dx\le CR^{-2}\int_{B_{3R/2}(y)\setminus B_{R/4}(y)}\abs{\vec v_\epsilon}^2\,dx.
\end{equation}
Then by combining \eqref{K1229.eq3a} and \eqref{1229.eq3b}, we find that 
\begin{equation}		\label{1229.eq3c}
\int_\Omega \eta^2|D\vec v_\epsilon|^2\,dx\le C R^{-2}\int_{B_{3R/2}(y)\setminus B_{R/4}(y)} \abs{\vec v_\epsilon}^2\,dx \le CR^{2-n},
\end{equation}
where we used Lemma \ref{0929-thm1} in the last inequality.
Also, by using the fact that 
\[
\norm{\eta \vec v_\epsilon}_{L^{2n/(n-2)}(\Omega)}\le C\norm{D(\eta \vec v_\epsilon)}_{L^2(\Omega)}\le C\norm{\eta D\vec v_\epsilon}_{L^2(\Omega)}+C\norm{D\eta \vec v_\epsilon}_{L^2(\Omega)},
\]
the inequality \eqref{1229.eq3c} implies 
\[
\norm{\vec G_\epsilon(\cdot,y)}_{L^{2n/(n-2)}(\Omega\setminus B_R(y))}+\norm{D\vec G_\epsilon(\cdot,y)}_{L^2(\Omega\setminus B_R(y))}\le CR^{(2-n)/2}.
\]
This together with \eqref{K1230.eq1} gives \eqref{1229.eq2b} for $R\in (0,d_y/3]$.

Now, let $A_t=\set{x\in \Omega:\abs{\vec G_\epsilon(x,y)}>t}$ and choose $t=R^{2-n}>d_y^{2-n}$.
Then by \eqref{1229.eq2b}, we have 
\begin{equation*}
\abs{A_t\setminus B_R(y)}\le t^{-2n/(n-2)}\int_{A_t\setminus B_R(y)}\abs{\vec G_\epsilon(x,y)}^{2n/(n-2)}\,dx\le C t^{-n/(n-2)}.
\end{equation*}
From this inequality and the fact that $\abs{A_t\cap B_R(y)}\le CR^n=Ct^{-n/(n-2)}$, we get \eqref{0103.eq3}. 
Let us fix $q\in [1,n/(n-2))$.
Note that 
\begin{align}
\nonumber
\int_{B_R(y)}\abs{\vec G_\epsilon(x,y)}^q\,dx&=\int_{B_R(y)\cap A_t^c}\abs{\vec G_\epsilon(x,y)}^q\,dx+\int_{B_R(y)\cap A_t}\abs{\vec G_\epsilon(x,y)}^q\,dx\\
\label{0103.eq5}
&\le C R^{(2-n)q+n}+\int_{A_t}\abs{\vec G_\epsilon(x,y)}^q\,dx,
\end{align}
where $t=R^{2-n}>d^{2-n}_y$.
From \eqref{0103.eq3} it follows that 
\begin{align}
\nonumber
\int_{A_t}\abs{\vec G_\epsilon(x,y)}^q\,dx&=q\int_0^\infty s^{q-1}\bigabs{\set{x\in \Omega:\abs{\vec G_\epsilon(x,y)}>\max(t,s)}}\,ds\\
\nonumber
&\le C_q t^{-n/(n-2)}\int_0^t s^{q-1}\,ds+C_q\int_t^\infty s^{q-1-n/(n-2)}\,ds\\
\label{0103.eq5a}
&\le C_qR^{(2-n)q+n},
\end{align}
where $C_q=C_q(n,\lambda,K_0, K_1,R_1,\mu,A_1,q)$.
Therefore, by combining \eqref{0103.eq5} and \eqref{0103.eq5a}, we obtain \eqref{1229.eq2c}.
Moreover, by utilizing \eqref{1229.eq2b}, and following the same steps as in the above, we get \eqref{0103.eq3a} and \eqref{1229.eq2d}.

It only remains to establish \eqref{0103.eq6}.
From H\"older's inequality, we only need to prove the inequality with $q\in (1,n/(n-1))$.
Let $q\in (1,n/(n-1))$ and $q'=q/(q-1)$, and denote 
\[
w:= \sgn (\pi_\epsilon) \abs{\pi_\epsilon}^{q-1}.
\]
Then we have  
\[
w\in L^{q'}(\Omega), \quad n<q'<\infty.
\]
Therefore by Remark \ref{K0122.rmk2} and  the  Sobolev inequality,  there exists a function $\vec \phi\in W^{1,q'}_0(\Omega)^n$ such that 
\begin{equation}		\label{0103.eq6b}
\begin{aligned}
&\dv \vec \phi=w-(w)_{\Omega} \,\text{ in }\, \Omega,\\
&
\norm{\vec \phi}_{L^\infty(\Omega)}\le C\|D\vec \phi\|_{L^{q'}(\Omega)}\le C\norm{w}_{L^{q'}(\Omega)}.
\end{aligned}
\end{equation}
We observe  that
\begin{equation}		\label{0103.eq6c}
\int_\Omega \pi_\epsilon \dv \vec \phi\,dx=\int_{\Omega}\pi_\epsilon (w-(w)_{\Omega})\,dx
=\int_{\Omega}\pi_\epsilon w\,dx=\int_{\Omega}\abs{w}^{q'}\,dx.
\end{equation}
By setting $\vec \varphi=\vec \phi$ in \eqref{0927.eq2a}, we get from \eqref{0103.eq6b} and \eqref{0103.eq6c} that 
\begin{equation}		\label{0105.eq2}
\int_{\Omega}\abs{w}^{q'}\,dx\le C\big(1+\norm{D\vec v_\epsilon}_{L^q(\Omega)}\big)\norm{w}_{L^{q'}(\Omega)}.
\end{equation}
Notice from \eqref{1229.eq2b} and \eqref{1229.eq2d} that 
\[
\norm{D\vec v_\epsilon}_{L^q(\Omega)}\le C_{y,q}
\]
for all $\varepsilon>0$, where $C_{y,q}=C_{y,q}(n,\lambda,K_0, K_1,R_1, \mu,A_1,q,d_y)$.
This together with \eqref{0105.eq2} gives \eqref{0103.eq6}.
The lemma is proved.
\end{proof}

\subsubsection{Construction of the Green function}		\label{0108.sec2}
Let $y\in \Omega$  be fixed, but arbitrary.
Notice from Lemma \ref{1229-lem1} and the weak compactness theorem that there exist a sequence $\set{\epsilon_\rho}_{\rho=1}^\infty$ tending to zero and functions $\vec G(\cdot,y)$ and $\hat{\vec G}(\cdot,y)$ such that 
\begin{align}
\nonumber
&\vec G_{\epsilon_\rho}(\cdot,y) \rightharpoonup \vec G(\cdot,y) \quad \text{weakly in }\, W^{1,2}(\Omega \setminus \overline{B_{d_y/2}(y)})^{n\times n},\\
\label{0103.e1a}
&\vec G_{\epsilon_\rho}(\cdot,y) \rightharpoonup \hat{\vec G}(\cdot,y) \quad \text{weakly in }\, W^{1,q}(B_{d_y}(y))^{n\times n},
\end{align}
where $q\in (1,n/(n-1))$.
Since $\vec G(\cdot,y)\equiv \hat{\vec G}(\cdot,y)$ on $B_{d_y}(y)\setminus \overline{B_{d_y/2}(y)}$, we shall extend $\vec G(\cdot,y)$ to entire $\Omega$ by setting $\vec G(\cdot,y)\equiv \hat{\vec G}(\cdot,y)$ on $\overline{B_{d_y/2}(y)}$.
By applying a diagonalization process and passing to a subsequence, if necessary, we may assume that
\begin{equation}		\label{0103.e1b}
\vec G_{\epsilon_\rho}(\cdot,y) \rightharpoonup \vec G(\cdot,y) \quad \text{weakly in }\, W^{1,2}(\Omega \setminus \overline{B_R(y)})^{n\times n}, \quad \forall R\in (0,d_y].
\end{equation}
Indeed, if we consider a sequence $\{R_i\}_{i=1}^\infty$ satisfying $R_i\in (0, d_y]$ and $R_i \searrow 0$, then for each $i\in \{1,2,\ldots\}$, there exists  a subsequence of $\{\vec G_{\varepsilon_\rho}(\cdot,y)\}$, denoted by $\big\{\vec G_{\varepsilon_{\rho_{i,j}}}(\cdot,y)\big\}$, such that 
$$
\big\{\vec G_{\varepsilon_{\rho_{i+1,j}}}(\cdot,y)\big\}\subset \big\{\vec G_{\varepsilon_{\rho_{i,j}}}(\cdot,y)\big\}
$$
and 
$$
\vec G_{\epsilon_{\rho_{i,j}}}(\cdot,y) \rightharpoonup \vec G(\cdot,y) \quad \text{weakly in }\, W^{1,2}(\Omega \setminus \overline{B_{R_i}(y)})^{n\times n} \quad \text{as }\, j\to \infty.
$$
Taking the subsequence as $\big\{\vec G_{\varepsilon_{\rho_{i,i}}}(\cdot,y)\big\}$, we see that \eqref{0103.e1b} holds.
By \eqref{0103.eq6},  there exists a function $\vec \Pi(\cdot,y)\in L^q_0(\Omega)^n$ such that, by passing to a subsequence, 
\begin{equation}		\label{0103.eq1c}
\vec \Pi_{\epsilon_\rho}(\cdot,y) \rightharpoonup {\vec \Pi}(\cdot,y) \quad \text{weakly in }\, L^q(\Omega)^n.
\end{equation}

We shall now claim that $(\vec G(x,y),\vec \Pi(x,y))$ satisfies the properties  a) -- c) in Definition \ref{0110.def} so that $(\vec G(x,y),\vec \Pi(x,y))$ is indeed the Green function for $\eqref{dp}$.
Notice from \eqref{0103.e1b} that for any $\zeta\in C^\infty_0(\Omega)$ satisfying $\zeta\equiv 1$ on $B_R(y)$, where $R\in (0,d_y)$, we have
\[
(1-\zeta)\vec G_{\epsilon_\rho}(\cdot,y)\rightharpoonup (1-\zeta)\vec G(\cdot,y) \quad \text{weakly in }\, W^{1,2}(\Omega)^{n\times n}.
\]
Since $W^{1,2}_0(\Omega)$ is weakly closed in $W^{1,2}(\Omega)$, we have $(1-\zeta)\vec G(\cdot,y)\in W^{1,2}_0(\Omega)^{n\times n}$, and thus the property a) is verified.
Let $\eta$ be a smooth cut-off function satisfying $\eta\equiv 1$ on $B_{d_y/2}(y)$ and $\supp \eta\subset B_{d_y}(y)$.
Then by \eqref{0927.eq2a}, \eqref{0103.e1a} -- \eqref{0103.eq1c}, we obtain for $\vec \varphi\in C^\infty_0(\Omega)^n$ that  
\begin{align}
\nonumber
\varphi^k(y)=&\lim_{\rho\to \infty}\fint_{\Omega_{\epsilon_\rho}(y)} \varphi^k\\
\nonumber
=&\lim_{\rho\to \infty}\left(\int_\Omega a^{ij}_{\alpha\beta}D_\beta G^{jk}_{\epsilon_\rho}(\cdot,y)D_\alpha(\eta \varphi^i)+\int_\Omega a^{ij}_{\alpha\beta}D_\beta G^{jk}_{\epsilon_\rho}(\cdot,y)D_\alpha((1-\eta) \varphi^i)\right)\\
\nonumber
&-\lim_{\rho\to \infty}\int_\Omega \Pi^k_{\epsilon_\rho}(\cdot,y)\dv  \vec\varphi\\
\nonumber
=&\int_\Omega a^{ij}_{\alpha\beta}D_\beta G^{jk}(\cdot,y)D_\alpha \varphi^i-\int_\Omega \Pi^k(\cdot,y)\dv \vec \varphi.
\end{align}
Similarly, we get 
\[
\int_\Omega \phi(x) \dv_x \vec G(x,y)\,dx=0, \quad \forall \phi\in C^\infty(\Omega).
\] 
From the above two identity, the property b) is satisfied.
Finally, if $(\vec u,p)\in W^{1,2}_0(\Omega)^n\times L^2_0(\Omega)$ is the weak solution of the problem \eqref{dps},
then by setting $\vec \varphi$ to be the $k$-th column of $\vec G_{\epsilon_\rho}(\cdot,y)$ in \eqref{1218.eq1} and setting $\vec \varphi=\vec u$ in \eqref{0927.eq2a}, we have (see e.g., Eq. \eqref{0105.eq0})
\begin{equation} 		\label{160906@eq10}		
\fint_{\Omega_{\epsilon_\rho}(y)}  \vec u=\int_\Omega \vec G_{\epsilon_\rho}(\cdot,y)^{\tr}\vec f-\int_\Omega D_\alpha \vec G_{\epsilon_\rho}(\cdot,y)^{\tr}\vec f_\alpha-\int_\Omega \vec \Pi_{\epsilon_\rho}(\cdot,y)g.
\end{equation}
By letting $\rho\to \infty$ in the above identity, we find that $(\vec G(x,y),\vec \Pi(x,y))$ satisfies the property c) in Definition \ref{0110.def}.

Next, let $y\in \Omega$ and $R\in (0,d_y]$.
Let $\vec v$ and $\vec v_{\epsilon}$ be the $k$-th column of $\vec G(\cdot,y)$ and $\vec G_\epsilon(\cdot,y)$, respectively.
Then for any $\vec g\in C^\infty_0(B_R(y))^n$, we obtain by \eqref{1229.eq2c} and \eqref{0103.e1a} that 
\begin{equation*}
\Abs{\int_{B_R(y)}  \vec v\cdot \vec g\,dx}=\lim_{\rho\to \infty}\Abs{\int_{B_R(y)}\vec v_{\epsilon_\rho}\cdot \vec g\,dx}\le C_q R^{2-n+n/q}\norm{\vec g}_{L^{q'}(B_R(y))},
\end{equation*}
where $q\in [1,n/(n-2))$ and $q'=q/(q-1)$.
Therefore, by a duality argument, we obtain the estimate iv) in Theorem \ref{1226.thm1}.
Similarly, from  Lemma \ref{1229-lem1}, \eqref{0103.e1a}, and \eqref{0103.e1b}, we have the estimates i) and v) in the theorem.  
Also, ii) and iii) are deduced from i) in the same way as \eqref{0103.eq3} and \eqref{0103.eq3a} are deduced from \eqref{1229.eq2b}.
Therefore, $\vec G(x,y)$ satisfies the estimates i) -- v) in Theorem \ref{1226.thm1}. 
For $x, y\in \Omega$ satisfying $0<\abs{x-y}<d_y/2$, set $r:=\abs{x-y}/4$.
Notice from the property b) in Definition \ref{0110.def} that $(\vec G(\cdot,y), \vec \Pi(\cdot,y))$ satisfies 
\[
\left\{
\begin{aligned}
\sL\vec G(\cdot,y)+D\vec \Pi(\cdot,y)=0 &\quad \text{in }\, B_{r}(x),\\
\dv \vec G(\cdot,y)=0 &\quad \text{in }\, B_{r}(x).
\end{aligned}
\right.
\]
Then by Lemma \ref{1227.lem1} and H\"older's inequality, we have  
\begin{equation*}
\abs{\vec G(x,y)}\le Cr^{(2-n)/2}\norm{\vec G(\cdot,y)}_{L^{2n/(n-2)}(B_{2r}(x))}\le Cr^{(2-n)/2}\norm{\vec G(\cdot,y)}_{L^{2n/(n-2)}(\Omega \setminus B_{r}(y))}.
\end{equation*}
This together with the estimate i) in Theorem \ref{1226.thm1} implies 
\[
\abs{\vec G(x,y)}\le C\abs{x-y}^{2-n}, \quad 0<\abs{x-y}<d_y/2. 
\] 
\begin{lemma}		\label{0108.lem1}
For each compact set $K\subset \Omega \setminus \set{y}$, there is a subsequence of $\{\vec G_{\varepsilon_\rho}(\cdot,y)\}$ that converges to $\vec G(\cdot,y)$ uniformly on $K$.
\end{lemma}

\begin{proof}
Let $x\in \Omega$ and $R\in (0,d_x]$ satisfying $\overline{B_R(x)}\subset \Omega \setminus \set{y}$.
Notice that  there exists $\epsilon_B>0$ such that for $\epsilon<\epsilon_B$, we have 
\[
\left\{
\begin{aligned}
\sL\vec G_\epsilon(\cdot,y)+D\vec \Pi_\epsilon(\cdot,y)=0 &\quad \text{in }\, B_R(x),\\
\dv \vec G_\epsilon(\cdot,y)=0 &\quad \text{in }\, B_R(x).
\end{aligned}
\right.
\]
By $(\bf{A1})$ and \eqref{1229.eq2b}, $\set{\vec G_\epsilon(\cdot,y)}_{\epsilon\le\epsilon_B}$ is equicontinuous on $\overline{B_{R/2}(x)}$.
Also, it follows from Lemma \ref{1227.lem1} that $\set{\vec G_\epsilon(\cdot,y)}_{\epsilon\le\epsilon_B}$ is uniformly bounded on $\overline{B_{R/2}(x)}$.
By the Arzel\`a-Ascoli theorem, we obtain the desired conclusion.
\end{proof}

\subsubsection{Proof of the identity \eqref{1226.eq1a}}

For any $x\in \Omega$ and $\sigma>0$, we define the averaged Green function $(\vec G^*_{\sigma}(\cdot,x), \vec \Pi_\sigma^*(\cdot,x))$ for $\eqref{dps}$ by letting its $l$-th column to be the unique weak solution in $W^{1,2}_0(\Omega)^n\times L^2_0(\Omega)$ of the problem
\[
\left\{
\begin{aligned}
\sL^*\vec u+Dp=\frac{1}{\abs{\Omega_\sigma(x)}}1_{\Omega_\sigma(x)}\vec e_l &\quad \text{in }\, \Omega,\\
\dv \vec u=0 &\quad \text{in }\, \Omega,
\end{aligned}
\right.
\]
where $\vec e_l$ is the $l$-th unit vector in $\bR^n$.
Then by following the same argument as in Sections \ref{0108.sec1} and \ref{0108.sec2}, there exist a sequence $\set{\sigma_\nu}_{\nu=1}^\infty$ tending to zero and the Green function $(\vec G^*(\cdot,x), \vec \Pi^*(\cdot,x))$ for $\eqref{dps}$ 
satisfying the counterparts of \eqref{0103.e1a}, \eqref{0103.e1b}, \eqref{0103.eq1c}, and Lemma \ref{0108.lem1}.

Now, let $x,\,y \in \Omega$ and $x\neq y$.
We then obtain for $\epsilon\in (0,d_y]$ and $\sigma\in (0,d_x]$ that  
\begin{equation}		\label{jk.eq2b}
\fint_{B_\epsilon(y)} (G^*_{\sigma})^{kl}(\cdot,x)=\int_\Omega a^{ij}_{\alpha\beta}D_\beta G^{jk}_\epsilon(\cdot,y)D_\alpha \big((G^*_\sigma)^{il}(\cdot,x)\big)=\fint_{B_{\sigma}(x)}G_\epsilon^{lk}(\cdot,y).
\end{equation}
We define
\[
I^{kl}_{\rho,\nu}:=\fint_{B_{\epsilon_\rho}(y)}(G^*_{\sigma_{\nu}})^{kl}(\cdot,x)=\fint_{B_{\sigma_\nu}}G^{lk}_{\epsilon_\rho}(\cdot,y).
\]
Then by the continuity of $\vec G_{\epsilon_\rho}(\cdot,y)$ and Lemma \ref{0108.lem1}, we have 
\[
\lim_{\rho\to \infty}\lim_{\nu\to \infty}I^{kl}_{\rho,\nu}=\lim_{\rho\to \infty}G^{lk}_{\epsilon_\rho}(x,y)=G^{lk}(x,y).
\]
Similarly,  we get
\[
\lim_{\rho\to \infty}\lim_{\nu\to \infty}I^{kl}_{\rho,\nu}=\lim_{\rho\to \infty}\fint_{\Omega_{\epsilon_\rho}(y)}(G^*)^{kl}(\cdot,x)=(G^*)^{kl}(y,x).
\]
We have thus shown that 
\[
G^{lk}(x,y)=(G^*)^{kl}(y,x), \quad \forall x,\, y\in \Omega, \quad x\neq y,
\]
which gives the identity \eqref{1226.eq1a}.
Therefore, we get from \eqref{jk.eq2b} that  
\begin{align*}
 G^{lk}_\epsilon(x,y)&=\lim_{\nu\to \infty}\fint_{B_{\sigma_\nu}(x)}G^{lk}_\epsilon(\cdot,y)=\lim_{\nu\to\infty}\fint_{B_\epsilon(y)}(G^*_{\sigma_\nu})^{kl}(\cdot,x)\\
 &=\fint_{B_\epsilon(y)}(G^*)^{kl}(\cdot,x)=\fint_{B_\epsilon(y)}G^{lk}(x,\cdot), \quad \epsilon\in (0,d_y],
\end{align*}
and 
\begin{equation}		\label{0109.eq2a}
\lim_{\epsilon\to 0}G^{lk}_\epsilon(x,y)=G^{lk}(x,y), \quad \forall x,\,y\in \Omega,\quad x\neq y.
\end{equation}
The theorem is proved.
\hfill\qedsymbol

\subsection{Proof of Theorem \ref{0110.thm1}}
The proof is based on $L^q$-estimates for Stokes systems with $\VMO$ coefficients.
In this proof, we assume that $x_0\in \Omega$ and $0<R\le \min\{d_{x_0},1\}$, and denote $B_r=B_r(x_0)$ for $r>0$.

\begin{lemma}		\label{170112@lem1}
Let $q>n$, $0<\rho<r\le R\le 1$, and $(\vec v,b)\in W^{1,q}(B_r)^n\times L^q(B_r)$ satisfy
$$
\left\{
\begin{aligned}
\sL\vec v+Db=0 \quad \text{in }\, B_r,\\
\dv \vec v=0 \quad \text{in }\, B_r,
\end{aligned}
\right.
$$
where the coefficients of $\sL$ belong to the class of $\VMO$.
Then we have 
$$
\|D\vec v\|_{L^q(B_{\rho})}+\frac{1}{r-\rho}\|\vec v\|_{L^q(B_\rho)}\le \frac{C}{r-\rho}\left(\|D\vec v\|_{L^{nq/(n+q)}(B_{r})}+\frac{1}{r-\rho}\|\vec v\|_{L^{nq/(n+q)}(B_{r})}\right),
$$
where $C$ depends on $n$, $\lambda$, $q$, and the $\VMO$ modulus of the coefficients.
\end{lemma}

\begin{proof}
Let $\tau=(\rho+r)/2$ and  $\eta$ be a smooth function in $\bR^2$ such that 
$$
0\le \eta\le 1, \quad \eta\equiv 1 \,\text{ on }\, B_{\rho}, \quad \supp \eta\subset B_{\tau}, \quad |D\eta|\le C(r-\rho)^{-1}, 
$$
Denote $b_0=(b)_{B_r}$ and observe that  $(\eta\vec v,\eta (b-b_0))$ satisfies
$$
\left\{
\begin{aligned}
\sL(\eta \vec v)+D(\eta (b-b_0))=(b-b_0)D\eta-A_{\alpha\beta} D_\beta \vec v D_\alpha \eta-D_\alpha(A_{\alpha\beta}D_\beta \eta \vec v) &\quad \text{in }\, B_r,\\
\dv (\eta \vec v)=D\eta \cdot \vec v &\quad \text{in }\, B_r,\\
\eta \vec v=0 &\quad \text{on }\partial B_r.
\end{aligned}
\right.
$$
By Corollary \ref{0129.cor2} with scaling, we have
$$
\|D\vec v\|_{L^q(B_\rho)}\le \frac{C}{r-\rho}\big(\|b-b_0\|_{L^{nq/(n+q)}(B_r)}+\|D\vec v\|_{L^{nq/(n+q)}(B_r)}+\|\vec v\|_{L^q(B_\tau)}\big),
$$
where $C$ depends on $n$, $\lambda$, $q$, and the $\VMO$ modulus of the coefficients.
Note that 
\begin{equation}		\label{170112@eq10}
\|\vec v\|_{L^q(B_{r_1})}\le \frac{C}{r_1}\|\vec v\|_{L^{nq/(n+q)}(B_{r_1})}+C\|D\vec v\|_{L^{nq/(n+q)}(B_{r_1})}
\end{equation}
for $0<r_1\le r$.
Combining the above two estimates we have 
\begin{equation}		\label{170112@eq1}
\begin{aligned}
&\|D\vec v\|_{L^q(B_{\rho})}+\frac{1}{r-\rho}\|\vec v\|_{L^q(B_\rho)}\\
&\quad \le \frac{C}{r-\rho}\left(\|b-b_0\|_{L^{nq/(n+q)}(B_r)}+\|D\vec v\|_{L^{nq/(n+q)}(B_{r})}+\frac{1}{r-\rho}\|\vec v\|_{L^{nq/(n+q)}(B_{r})}\right).
\end{aligned}
\end{equation}

Set $s=nq/(n+q)$ and $\tilde{b}=\sgn(b-b_0)|b-b_0|^{s-1}\in L^{s/(s-1)}(B_r)$.
There exists $\vec \phi\in W^{1,s/(s-1)}_0(B_r)^n$ such that (see Remark \ref{K0122.rmk2})
$$
\dv \vec \phi= \tilde{b}-(\tilde{b})_{B_r} \, \text{ in }\, B_r, \quad \|D\vec \phi\|_{L^{s/(s-1)}(B_r)}\le C(n,q)\|\tilde{b}\|_{L^{s/(s-1)}(B_r)}.
$$
Using $\vec \phi$ as a test function, we obtain 
$$
\int_{B_r}|b-b_0|^s\,dx=\int_{B_r}(b-b_0)\dv \vec \phi\,dx=\int_\Omega A_{\alpha\beta}D_\beta \vec v\cdot D_\alpha \vec \phi\,dx,
$$
which implies that 
$$
\|b-b_0\|_{L^s(B_r)}^s\le C(n,\lambda,q)\|D\vec v\|_{L^s(B_r)}\|b-b_0\|_{L^s(B_r)}^{s-1}.
$$
From this together with \eqref{170112@eq1}, we get the desired estimate.
\end{proof}

Now we are ready to prove Theorem  \ref{0110.thm1}.
Let $(\vec u, p)\in W^{1,2}(B_R)^n\times L^2(B_R)$ satisfy \eqref{160907@eq2}.
Let $q>n$, $0<r\le R$, and $\rho=r/4$.
Set
$$
q_i=\frac{nq}{n+qi}, \quad r_i=\rho+\frac{ri}{4m}, \quad i\in \{0,\ldots,m\},
$$
where $m$ is the smallest integer such that $m\ge n(1/2-1/q)$.
Then by applying Lemma \ref{170112@lem1} iteratively, we see that $(\vec u, p)\in W^{1,q}(B_\rho)^n\times L^q(B_\rho)$ and
$$
\|D\vec u\|_{L^q(B_\rho)}+\frac{4m}{r}\|\vec u\|_{L^q(B_\rho)}\le \left(\frac{Cm}{r}\right)^m\left(\|D\vec u\|_{L^{q_m}(B_{r_m})}+\frac{4m}{r}|\vec u\|_{L^{q_m}(B_{r_m})}\right).
$$
Using H\"older's inequality and Lemma \ref{1006@lem1}, we have 
\begin{align*}
\|D\vec u\|_{L^q(B_{r/4})}+\frac{1}{r}\|\vec u\|_{L^q(B_{r/4})}&\le \left(\frac{Cmr}{r}\right)^mr^{n(1/q-1/2)}\left(\|D\vec u\|_{L^{2}(B_{r/2})}+\frac{1}{r}\|\vec u\|_{L^{2}(B_{r/2})}\right)\\
&\le \frac{Cr^{n(1/q-1/2)}}{r}\|\vec u\|_{L^2(B_r)}.
\end{align*}
By the Sobolev inequality with scaling, we get 
$$
[\vec u]_{C^{1-n/q}(B_{r/4}(x_0))}\le Cr^{-1+n/q}\left(\fint_{B_r(x_0)}|\vec u|^2\,dx\right)^{1/2},
$$
where $C$ depends on $n$, $\lambda$, $q$, and the $\VMO$ modulus of the coefficients.
Since the above inequality holds for all $x_0\in \Omega$ and  $0<r\le R\le \min\{d_{x_0},1\}$, we conclude that 
$$
[\vec u]_{C^{1-n/q}(B_{R/2}(x_0))}\le Cr^{-1+n/q}\left(\fint_{B_R(x_0)}|\vec u|^2\,dx\right)^{1/2}.
$$
This completes the proof of Theorem \ref{0110.thm1}.
\hfill\qedsymbol

\subsection{Proof of Theorem \ref{1226.thm2}}
For $y\in \Omega$ and $\epsilon>0$, let $(\vec G_{\varepsilon}(\cdot,y), \vec \Pi_{\varepsilon}(\cdot,y))$ be the averaged Green function on $\Omega$ as constructed in Section \ref{0108.sec1}, and let $\vec G^{\cdot k}_{\varepsilon}(\cdot,y)$ be the $k$-th column of $\vec G_\varepsilon(\cdot,y)$.
Recall that $(\vec G_{\varepsilon}^{\cdot k}(\cdot,y), \Pi^k_\varepsilon(\cdot,y))$ satisfies
\[
\left\{
\begin{aligned}
\sL\vec G_\epsilon^{\cdot k}(\cdot,y)+D\vec \Pi_\epsilon^k(\cdot,y)=\vec g_k &\quad \text{in }\, \Omega,\\
\dv \vec G_\epsilon^{\cdot k}(\cdot,y)=0 &\quad \text{in }\, \Omega,
\end{aligned}
\right.
\]
where 
$$
\vec g_k=\frac{1}{|\Omega_{\varepsilon}(y)|}1_{\Omega_{\varepsilon}(y)}\vec e_k.
$$
By $(\bf{A2})$, we obtain for any $x_0\in \Omega$ and $0<r<\diam\Omega$ that 
$$
\|\vec G^{\cdot k}_\varepsilon(\cdot,y)\|_{L^\infty(\Omega_{r/2}(x_0))}\le A_2\left(r^{-n/2}\|\vec G^{\cdot k}_\varepsilon(\cdot,y)\|_{L^2(\Omega_r(x_0))}+r^2\|\vec g_k\|_{L^\infty(\Omega_r(x_0))}\right).
$$
Applying a standard argument (see, for instance, \cite[pp. 80-82]{MR1239172}), we have 
$$
\|\vec G^{\cdot k}_\varepsilon(\cdot,y)\|_{L^\infty(\Omega_{r/2}(x_0))}\le C\left(r^{-n}\|\vec G^{\cdot k}_\varepsilon(\cdot,y)\|_{L^1(\Omega_r(x_0))}+r^2\|\vec g_k\|_{L^\infty(\Omega_r(x_0))}\right),
$$
where $C=C(n,A_2)$.
We remark that if $B_r(x_0)\cap B_\varepsilon(y)=\emptyset$, then 
\begin{equation}		\label{170418@eq1a}
\|\vec G^{\cdot k}_\varepsilon(\cdot,y)\|_{L^\infty(\Omega_{r/2}(x_0))}\le Cr^{-n}\|\vec G^{\cdot k}_\varepsilon(\cdot,y)\|_{L^1(\Omega_r(x_0))}.
\end{equation}

Next, let $y\in \Omega$ and $R\in (0,\diam\Omega)$.
Assume that $\vec f\in L^\infty(\Omega)^n$  with $\supp f\subset\Omega_R(y)$.
Let $(\vec u,p)\in W^{1,2}_0(\Omega)^n\times L^2_0(\Omega)$ be the weak solution of the problem
\[
\left\{
\begin{aligned}
\sL^*\vec u+Dp=\vec f &\quad \text{in }\, \Omega,\\
\dv \vec u=0 &\quad \text{in }\, \Omega.
\end{aligned}
\right.
\]
By $(\bf{A2})$, the Sobolev inequality, and \eqref{1227.eq1}, we have 
$$
\|\vec u\|_{L^\infty(\Omega_{R/2}(y))}\le A_2\left(R^{-n/2}\|\vec u\|_{L^2(\Omega_R(y))}+R^2\|\vec f\|_{L^\infty(\Omega_R(y))}\right)\le CR^2\|\vec f\|_{L^\infty(\Omega_R(y))},
$$
where $C=C(n,\lambda,K_0, K_1,R_1,A_2)$.
Using this together with the fact that  (see, for instance, \eqref{160906@eq10})
$$
\fint_{\Omega_{\varepsilon}(y)}u^k\,dx=\int_{\Omega_R(y)}G^{ik}_{\varepsilon}(\cdot,y)f^i\,dx,
$$
we have 
$$
\left|\int_{\Omega_R(y)}G^{ik}_\varepsilon(\cdot,y)f^i\,dx\right|\le CR^2\|\vec f\|_{L^\infty(\Omega_R(y))}
$$
for all $0<\varepsilon<R/2$ and $\vec f\in L^\infty(\Omega_R(y))^n$.
Taking 
$$
f^i(x)=1_{\Omega_R(y)} \sgn (G^{ik}_\varepsilon(x,y)),
$$
 we have 
\begin{equation}		\label{170418@eq1b}
\norm{\vec G^{\cdot k}_\epsilon(\cdot,y)}_{L^1(\Omega_R(y))}\le CR^2,  \quad \forall \epsilon\in (0,R/2).
\end{equation}

Now we are ready to prove the theorem.
Let $x,\,y\in \Omega$ and $x\neq y$ and take $R=3r=3\abs{x-y}/2$.
Then by \eqref{170418@eq1a} and \eqref{170418@eq1b}, we obtain for $\epsilon\in (0,r)$ that 
\begin{align}	
\nonumber
\abs{\vec G_\epsilon(x,y)}&\le Cr^{-n}\norm{\vec G_\epsilon(\cdot,y)}_{L^1(\Omega_r(x))}\le CR^{-n}\norm{\vec G_\epsilon(\cdot,y)}_{L^1(\Omega_R(y))}\le CR^{2-n},
\end{align}
where $C=C(n,\lambda,K_0, K_1,R_1,A_2)$.
Therefore, by letting $\epsilon\to 0$ and using \eqref{0109.eq2a}, we obtain that
\[
\abs{\vec G(x,y)}\le C\abs{x-y}^{2-n}.
\]
The theorem is proved.
\hfill\qedsymbol

\subsection{Proof of Theorem \ref{0110.thm2}}		\label{0304.sec1}
Let $(\vec u, p)\in W^{1,2}_0(\Omega)^n\times L^2_0(\Omega)$ be the weak solution of 
\eqref{170418@eq1}.
By Corollary \ref{0129.cor2}, $\vec u$ is H\"older continuous.
To prove the theorem, we first consider the localized estimates for Stokes systems as below.

For $y\in \Omega$ and $r>0$, we denote $B_r=B_r(y)$ and $\Omega_r=\Omega_r(y)$.\\
\\
{\bf{Step 1.}}
Let  $n/(n-1)<q\le t$, $0<\rho<r<\tau$, and $\eta,\,\zeta$ be smooth functions in $\bR^n$ satisfying
\begin{align*}
0\le \eta\le 1, \quad \eta\equiv 1 \, \text{ on }\, B_\rho, \quad \supp \eta\subset B_r, \quad |D\eta|\le C(r-\rho)^{-1},\\
0\le \zeta\le 1, \quad \zeta\equiv 1 \, \text{ on }\, B_r, \quad \supp \zeta\subset B_\tau, \quad |D\zeta|\le C(\tau-r)^{-1}.
\end{align*}
Then $(\eta\vec u,\eta p)$ is the weak solution of the problem
\begin{equation*}		
\left\{
\begin{aligned}
\sL(\eta \vec u)+D(\eta p)=\eta \vec f+pD\eta- \vec\Psi-D_\alpha\vec\Phi_\alpha &\quad \text{in }\Omega,\\
\dv \eta \vec u=D\eta \cdot  \vec u &\quad \text{in }\Omega,\\
\eta \vec u=0 &\quad \text{on }\, \partial \Omega,
\end{aligned}
\right.
\end{equation*}
where 
\[
\vec \Psi=A_{\alpha\beta}D_\beta \vec uD_\alpha \eta \quad \text{and}\quad \vec \Phi_\alpha=A_{\alpha\beta}D_\beta \eta \vec u.
\]
By Corollary \ref{0129.cor2}, we have 
\begin{multline*}
\norm{\eta p-(\eta p)_{\Omega}}_{L^{q}(\supp \eta\cap \Omega)}+\norm{D\vec u}_{L^{q}(\Omega_\rho)}\le \frac{C}{r-\rho}\big(\|p\|_{L^{nq/(n+q)}(\Omega_r)}+\norm{D\vec u}_{L^{nq/(n+q)}(\Omega_r)}\big)\\
+C\left( r^{1+n/q}\norm{\vec f}_{L^\infty(\Omega_r)}+\frac{1}{r-\rho}\norm{\vec u}_{L^{q}(\Omega_r)}\right).
\end{multline*}
Using the fact that  
\begin{align*}
\|p\|_{L^{nq/(n+q)}(\Omega_r)}&=\|\zeta p-(\zeta p)_\Omega+(\zeta p)_\Omega\|_{L^{nq/(n+q)}(\Omega_r)}\\
&\le \|\zeta p-(\zeta p)_\Omega\|_{L^{nq/(n+q)}(\supp \zeta\cap \Omega)}+C(n,q)\tau^{1+n/q}|(\zeta p)_\Omega|,
\end{align*}
we have
\begin{align}			
\nonumber
&\norm{\eta p-(\eta p)_{\Omega}}_{L^{q}(\supp \eta\cap \Omega)}+\norm{D\vec u}_{L^{q}(\Omega_\rho)}\\
\nonumber
&\le \frac{C}{r-\rho}\big(\|\zeta p-(\zeta p)_{\Omega}\|_{L^{nq/(n+q)}(\supp \zeta\cap \Omega)}+\norm{D\vec u}_{L^{nq/(n+q)}(\Omega_r)}\big)\\
\label{170112@eq5b}
&\quad +\frac{C}{r-\rho}\tau^{1+n/q}|(\zeta p)_\Omega|+C\left( r^{1+n/q}\norm{\vec f}_{L^\infty(\Omega_r)}+\frac{1}{r-\rho}\norm{\vec u}_{L^{q}(\Omega_r)}\right),
\end{align}
where $C$ depends on $n$, $\lambda$, $K_0$, $R_1$, $q$, and the $\VMO$ modulus of the coefficients.\\
\\
{\bf{Step 2.}}
Let $t>n$ and $0<\rho<r<\diam\Omega$.
Set 
$$
t_i=\frac{nt}{n+ti}, \quad r_i=\rho+(r-\rho)i/m, \quad i\in\{0,\ldots,m+1\},
$$
where $m$ is the smallest integer such that $t_m\le 2$.
Let $\eta_{r_i}$, $i\in\{0,\ldots,m\}$, be smooth functions in $\bR^n$ satisfying
$$
0\le \eta_{r_i}\le 1, \quad \eta_{r_i}\equiv 1 \, \text{ on }\, B_{r_i}, \quad \supp \eta_{r_i}\subset B_{r_{i+1}}, \quad |D\eta_{r_i}|\le \frac{C(n,t)}{r-\rho}.
$$
Applying \eqref{170112@eq5b} iteratively, we have 
\begin{align*}
\|D\vec u\|_{L^{t_0}(\Omega_{r_0})}&\le C^m\left(\frac{m}{r-\rho}\right)^m\big(\|\eta_{r_m} p-(\eta_{r_m} p)_\Omega\|_{L^{t_m}(\supp \eta_{r_m}\cap \Omega)}+\|D\vec u\|_{L^{t_m}(\Omega_{r_m})}\big)\\
&\quad +\sum_{i=1}^m C^i\left(\frac{m}{r-\rho}\right)^{i}r^{1+n/t_{i-1}}|(\eta_{r_i} p)_\Omega|\\
&\quad +\sum_{i=1}^mC^i\left(\frac{m}{r-\rho}\right)^{i-1}\left(r^{1+n/t_{i-1}}\|\vec f\|_{L^\infty(\Omega_{r})}+\frac{m}{r-\rho}\|\vec u\|_{L^{t_{i-1}}(\Omega_r)}\right).
\end{align*}
Hence, by H\"older's inequality we obtain 
\begin{align*}
\|D\vec u\|_{L^{t}(\Omega_{\rho})}&\le C_0\left(\frac{r}{r-\rho}\right)^mr^{n(1/t-1/2)}\big(\|p\|_{L^{2}(\Omega_r)}+\|D\vec u\|_{L^{2}(\Omega_{r})}\big)\\
&\quad +C_0\left(\frac{r}{r-\rho}\right)^mr^{1+n/t}\|\vec f\|_{L^\infty(\Omega_r)}+C_0\left(\frac{r}{r-\rho}\right)^mr^{-1}\|\vec u\|_{L^t(\Omega_r)},
\end{align*}
where $C_0$ depends on $n$, $\lambda$, $K_0$, $R_1$, and the $\VMO$ modulus of the coefficients.
By taking $\rho=r/2$, we have 
$$
\|D\vec u\|_{L^{t}(\Omega_{r/2})}\le C_0r^{n(1/t-1/2)}\big(\|p\|_{L^{2}(\Omega_r)}+\|D\vec u\|_{L^{2}(\Omega_{r})}\big) +C_0\big(r^{1+n/t}\|\vec f\|_{L^\infty(\Omega_r)}+r^{-1}\|\vec u\|_{L^t(\Omega_r)}\big).
$$
We apply Caccioppoli's inequality (see, for instance, \cite{MR2027755}) to the above estimate to get 
\begin{equation}	\label{170112@eq8}
\|D\vec u\|_{L^{t}(\Omega_{r/4})}\le C_0\big(r^{1+n/t}\|\vec f\|_{L^\infty(\Omega_r)}+r^{-1}\|\vec u\|_{L^t(\Omega_r)}\big).
\end{equation}
{\bf{Step 3.}}
We extend $\vec u$ to $\bR^n$ by setting $\vec u\equiv 0$ on $\bR^n\setminus \Omega$.
For $y\in \Omega$ and $0<r<\diam\Omega$, we obtain by \eqref{170112@eq10} and \eqref{170112@eq8} that 
\begin{align*}		
r^{-1}\|\vec u\|_{L^t(B_{r/4})}+\|D\vec u\|_{L^t(B_{r/4})}&\le C\big(r^{1+n/t}\|\vec f\|_{L^\infty(\Omega_r)}+r^{-1}\|\vec u\|_{L^t(B_r)}\big).
\end{align*}
Using this together with the Sobolev inequality, we have 
$$
\|\vec u\|_{L^\infty(B_{r/4})}\le C\big(r^{2}\|\vec f\|_{L^\infty(\Omega_r)}+r^{-n/t}\|\vec u\|_{L^t(B_r)}\big).
$$
Since the above estimate holds for any $y\in \Omega$ and $0<r<\diam\Omega$, by using a standard argument (see, for instance, \cite[pp. 80-82]{MR1239172}), we derive 
$$
\norm{\vec u}_{L^\infty(\Omega_{r/2})} \le C\big(r^2\norm{\vec f}_{L^\infty(\Omega_r)}+r^{-n/2}\norm{\vec u}_{L^2(\Omega_r)}\big).
$$
This completes the proof of Theorem \ref{0110.thm2}.
\hfill\qedsymbol

\section{$L^q$-estimates for the Stokes systems}		\label{sec_es}
In this section, we consider  the $L^q$-estimate for the solution to  
\begin{equation}		\label{0123.eq3}
\left\{
\begin{aligned}
\sL\vec u+Dp=\vec f+D_\alpha\vec f_\alpha&\quad \text{in }\, \Omega,\\
\dv \vec u= g &\quad \text{in }\, \Omega.
\end{aligned}
\right.
\end{equation}
We let $\Omega$ be a domain in $\bR^n$, where $n\ge 2$. 
We denote 
\begin{equation}		\label{160928@eq1}
U:=\abs{p}+\abs{D\vec u}\quad \text{and}\quad F:=\abs{\vec f}+\abs{\vec f_\alpha}+\abs{g},
\end{equation}
and we abbreviate $B_R=B_R(0)$ and $B_R^+=B_R^+(0)$, etc.

\subsection{Main results}		\label{123.sec1}

\begin{A} There is a constant $R_0\in (0,1]$ such that the following hold.
\begin{enumerate}[(a)]
\item
For any $x\in\overline{\Omega}$ and $R\in(0,R_0]$ so that either $B_R(x)\subset \Omega$ or $x\in \partial \Omega$, we have 
\begin{equation*}
\fint_{B_R(x)}\bigabs{A_{\alpha\beta}-(A_{\alpha\beta})_{B_R(x)}}\le \gamma.
\end{equation*}

\item ($\gamma$-Reifenberg flat domain)
For any $x\in \partial \Omega$ and $R\in (0,R_0]$, 
there is a spatial coordinate systems depending on $x$ and $R$ such that in this new coordinate system, we have 
\begin{equation*}		
\set{y:x_1+\gamma R<y_1}\cap B_R(x)\subset \Omega_R(x)\subset \set{y:x_1-\gamma R<y_1}\cap B_R(x). 
\end{equation*}
\end{enumerate}
\end{A}
 
\begin{theorem}		\label{0123.thm1}
Assume the condition $(\bf{D})$ in Section \ref{1006@sec2} and $\diam(\Omega)\le K_0$.
For $2<q<\infty$, there exists a constant $\gamma>0$, depending only on $n$, $\lambda$, and $q$, such that, under the condition $(\bf{A3}\,(\gamma))$, the following holds:
if $(\vec u,p)\in W^{1,q}_0(\Omega)^n\times L^q_0(\Omega)$ satisfies \eqref{0123.eq3},
then we have  
\begin{equation}		\label{1204.eq2}
\norm{p}_{L^q(\Omega)}+\norm{D\vec u}_{L^q(\Omega)}\le C\big(\norm{\vec f}_{L^q(\Omega)}+\norm{\vec f_\alpha}_{L^q(\Omega)}+\norm{g}_{L^q(\Omega)}\big),
\end{equation}
where $C=C(n,\lambda,K_0,q, A,R_0)$.
\end{theorem}

\begin{remark}
We remark that  $\gamma$-Reifenberg flat domains with a small constant $\gamma>0$ satisfy the condition $(\bf{D})$.
Indeed, $\gamma$-Reifenberg flat domains with sufficiently small $\gamma$ are John domains (and NTA-domains) that satisfy the condition $(\bf{D})$.
We refer to \cite{MR2186550, MR2263708, MR1446617} for the details.

\end{remark}

Since Lipschitz domains with a small Lipschitz constant are Refineberg flat, 
we obtain the following result from Theorem \ref{0123.thm1}.

\begin{corollary}		\label{0129.cor2}
Let $\Omega$ be a domain in $\bR^n$ with $\diam(\Omega)\le K_0$, where $n\ge 2$.
Assume that the coefficients of $\sL$ belong to the class of $\VMO$.
For $1<q<\infty$, there exists a constant $L=L(n,\lambda,q)>0$  such that, under the condition $(\bf{A0})$ with $R_1\in (0,1]$ and $K_1\in (0, L]$, the following holds:
if $q_1\in (1,\infty)$, $q_1\ge \frac{qn}{q+n}$, $\vec f\in L^{q_1}(\Omega)^n$, $\vec f_\alpha\in L^q(\Omega)^n$, and $g\in L^q_0(\Omega)$, there exists a unique solution $(\vec u,p)\in W^{1,q}_0(\Omega)^n\times L^q_0(\Omega)$ of the problem \eqref{0123.eq3}.
Moreover, we have
$$
\norm{p}_{L^q(\Omega)}+\norm{D\vec u}_{L^q(\Omega)}\le C\big(\norm{\vec f}_{L^{q_1}(\Omega)}+\norm{\vec f_\alpha}_{L^q(\Omega)}+\norm{g}_{L^q(\Omega)}\big),
$$
where the constant $C$ depends on $n$, $\lambda$, $K_0$, $R_1$, $q$, and the $\VMO$ modulus of the coefficients.
\end{corollary}

\begin{proof}
It suffices to prove the corollary with $\vec f=(f^1,\ldots,f^n)=0$.
Indeed, by the solvability of the divergence equation in Lipschitz domains, there exist $\vec \phi_i\in W^{1,q_1}_0(\Omega)^n$ such that 
$$
\dv{\vec \phi_i}=f^i-(f^i)_\Omega \quad \text{in }\, \Omega, \quad  \|D\vec \phi_i\|_{L^{q_1}(\Omega)}\le C\|f^i\|_{L^{q_1}(\Omega)},
$$
where $C=C(n,\lambda,K_0, R_1,q)$.
If we define $\vec \Phi_\alpha=(\Phi_\alpha^1,\ldots, \Phi^n_\alpha)$ by
$$
\Phi^i_\alpha(x)=\varphi^\alpha_i(x)+\frac{(f^i)_\Omega}{n}x_\alpha,
$$
then we have that 
$$
\sum_{\alpha=1}^nD_\alpha \vec \Phi_\alpha=\vec f
$$
and 
$$
\|\vec \Phi_\alpha \|_{L^q(\Omega)}\le C\|D\vec\Phi_\alpha\|_{L^{q_1}(\Omega)}\le C\|\vec f\|_{L^{q_1}(\Omega)}.
$$

Due to  Lemma \ref{122.lem1}, it is enough to consider the case $q\neq 2$.\\
\\
{\bf{Case 1.}} $q>2$.
Let $\gamma=\gamma(n,\lambda,q)$ and $M=M(n,q)$ be constants in Theorem \ref{0123.thm1} and \cite[Theorem 2.1]{MR1313554}, respectively.
Set $L=\min\{\gamma, M\}$.
If $K_1\in (0,L]$, 
then by Theorem \ref{0123.thm1}, the method of continuity, and the $L^q$-solvability of the Stokes systems with simple coefficients (see \cite[Theorem 2.1]{MR1313554}), 
there exists a unique solution $(\vec u,p)\in W^{1,q}_0(\Omega)^n\times L^q_0(\Omega)$ of the problem \eqref{0123.eq3} with $\vec f=0$.\\
\\
{\bf{Case 2.}} $1<q<2$.
We use the duality argument.
Set $q_0=\frac{q}{q-1}$, and let $L=L(n,\lambda,q_0)$ and $M=M(n,q)$ be constants from Case 1 and \cite[Theorem 2.1]{MR1313554}, respectively.
Assume that $K_1\le L$ and  $(\vec u, p)\in W^{1,q}_0(\Omega)^n\times L^q_0(\Omega)$ satisfies \eqref{0123.eq3} with $\vec f=0$.
For $\vec h_\alpha\in L^{q_0}(\Omega)^n$, there exists $(\vec v,\pi)\in W^{1,q_0}_0(\Omega)^n\times L^{q_0}_0(\Omega)$ such that 
\begin{equation*}		
\left\{
\begin{aligned}
\sL^*\vec v+D\pi=D_\alpha\vec h_\alpha&\quad \text{in }\, \Omega,\\
\dv \vec v=0 &\quad \text{in }\, \Omega,
\end{aligned}
\right.
\end{equation*}
where $\sL^*$ is the adjoint operator of $\sL$.
Then we have 
\begin{align*}
\int D_\alpha \vec u\cdot \vec h_\alpha\,dx&=-\int_\Omega A_{\alpha\beta}D_\beta \vec u \cdot D_\alpha \vec v\,dx+\int_\Omega \pi \dv \vec u\,dx\\
&=\int_\Omega \vec f_\alpha \cdot D_\alpha \vec v\,dx+\int_\Omega \pi g\,dx,
\end{align*}
which implies that 
$$
\left|\int D_\alpha \vec u\cdot \vec h_\alpha\,dx\right|\le C\big(\|\vec f_\alpha \|_{L^q(\Omega)}+\|g\|_{L^q(\Omega)}\big)\|\vec h_\alpha\|_{L^{q_0}(\Omega)},
$$
where the constant $C$ depends on $n$, $\lambda$, $K_0$, $R_1$, $q$, and the $\VMO$ modulus of the coefficients.
Since $\vec h_\alpha$ was arbitrary, it follows that 
\begin{equation}		\label{1007@e1}
\norm{D\vec u}_{L^q(\Omega)}\le C\big(\norm{\vec f_\alpha }_{L^q(\Omega)}+\norm{g}_{L^q(\Omega)}\big).
\end{equation}
To estimate $p$, let $w\in L^{q_0}(\Omega)$ and $w_0=w-(w)_\Omega$.
Then by Remark \ref{K0122.rmk2}, there exists $\vec \phi\in W^{1,q_0}(\Omega)^n$ such that 
\[
\dv \vec \phi=w_0 \quad \text{in }\, \Omega, \quad \norm{\vec \phi}_{W^{1,q_0}(\Omega)}\le C\norm{w_0}_{L^{q_0}(\Omega)}.
\]
By testing $\vec \phi$ in \eqref{0123.eq3}, it is easy to see that 
\begin{align*}
\Abs{\int_\Omega pw\,dx}&=\Abs{\int_\Omega pw_0\,dx}\\
&\le C\left(\norm{D\vec u}_{L^q(\Omega)}+\norm{\vec f_\alpha}_{L^q(\Omega)}\right)\norm{w_0}_{L^{q_0}(\Omega)}\\
&\le C\left(\norm{D\vec u}_{L^q(\Omega)}+\norm{\vec f_\alpha}_{L^q(\Omega)}\right)\norm{w}_{L^{q_0}(\Omega)}.
\end{align*}
This together with \eqref{1007@e1} yields 
$$
\|p\|_{L^q(\Omega)}+\norm{D\vec u}_{L^q(\Omega)}\le C\big(\norm{\vec f_\alpha }_{L^q(\Omega)}+\norm{g}_{L^q(\Omega)}\big).
$$
Using the above $L^q$-estimate, the method of continuity, and the $L^q$-solvability of the Stokes systems with simple coefficients, 
there exists a unique solution $(\vec u,p)\in W^{1,q}_0(\Omega)^n\times L^q_0(\Omega)$ of the problem \eqref{0123.eq3} with $\vec f=0$.
\end{proof}

\subsection{Auxiliary results}

\begin{lemma}		\label{160924@lem1}
Recall the notation \eqref{160928@eq1}.
Suppose that the coefficients of $\sL$ are constants.
Let $k$ be a constant.
\begin{enumerate}[$(a)$]
\item
If $(\vec u, p)\in W^{1,2}(B_R)^n\times L^2(B_R)$ satisfies
\begin{equation*}		
\left\{
\begin{aligned}
\sL\vec u+Dp=0&\quad \text{in }\, B_{R},\\
\dv \vec u= k &\quad \text{in }\, B_{R},
\end{aligned}
\right.
\end{equation*}
then there exists a constant $C=C(n,\lambda)$ such that 
\begin{equation}		\label{160927@eq4}
\norm{U}_{L^\infty(B_{R/2})}\le CR^{-n/2}\norm{U}_{L^2(B_R)}+C\abs{k}.
\end{equation}
\item
If $(\vec u, p)\in W^{1,2}(B_R^+)^n\times L^2(B_R^+)$ satisfies 
\begin{equation*}		
\left\{
\begin{aligned}
\sL\vec u+Dp=0&\quad \text{in }\, B_{R}^+,\\
\dv \vec u= k &\quad \text{in }\, B_{R}^+,\\
\vec u=0 &\quad \text{on }\, B_R\cap \set{x_1=0},
\end{aligned}
\right.
\end{equation*}
then there exists a constant $C=C(n,\lambda)$ such that 
\begin{equation}		\label{160927@eq3}		
\norm{U}_{L^\infty(B_{R/2}^+)}\le CR^{-n/2}\norm{U}_{L^2(B_R^+)}+C\abs{k}.
\end{equation}
\end{enumerate}
\end{lemma}

\begin{proof}
The interior and boundary estimates for Stokes systems with variable coefficients were studied by Giaquinta \cite{MR0641818}.
The proof of the assertion (a) is the same as that of  \cite[Theorem 1.10, pp. 186--187]{MR0641818}.
See also the proof of \cite[Theorem 2.8, p. 207]{MR0641818} for the boundary estimate \eqref{160927@eq3}.
We note that in \cite{MR0641818}, he gives the complete proofs for the Neumann problem and  mentioned that the method works for other boundary value problem.
Regarding the Dirichlet problem, we need to impose a normalization condition for $p$ because $(\vec u, p+c)$ satisfies the same system for any constant $c\in \bR$. 
By this reason, the right-hand sides of the estimates \eqref{160927@eq4} and \eqref{160927@eq3} contain the $L^2$-norm of $p$.
For more detailed proof, one may refer to \cite{arXiv:1604.02690v2}.
Their methods are general enough to allow the coefficients to be measurable in one direction and gives more precise information on the dependence of the constant $C$.
\end{proof}

\begin{theorem}		\label{123.thm1}
Let $2<\nu<q<\infty$ and $\nu'=2\nu/(\nu-2)$.
Assume $(\vec u,p)\in W^{1,q}_0(\Omega)^n\times L^q_0(\Omega)$ satisfies 
\begin{equation*}		
\left\{
\begin{aligned}
\sL\vec u+Dp=\vec f+D_\alpha\vec f_\alpha&\quad \text{in }\, \Omega,\\
\dv \vec u= g &\quad \text{in }\, \Omega,
\end{aligned}
\right.
\end{equation*}
where $\vec f,\, \vec f_\alpha\in L^2(\Omega)^n$ and $g\in L^2_0(\Omega)$.

\begin{enumerate}[(i)]
\item
Suppose that $(\bf{A3}\,(\gamma))$ $(a)$ holds at $0\in \Omega$ with $\gamma>0$.
Then, for $R\in (0, \min(R_0,d_0)]$, where $d_0=\dist(0,\partial \Omega)$, $(\vec u,p)$ admits a decomposition 
\[
(\vec u,p)=(\vec u_1,p_1)+(\vec u_2,p_2) \quad \text{in }\,  B_R,
\]
and we have 
\begin{align}		\label{123.eq1}
(U_1^2)^{1/2}_{B_R}&\le C\left(\gamma^{1/\nu'}(U^\nu)^{1/\nu}_{B_R}+(F^2)^{1/2}_{B_R}\right),\\
\label{123.eq1a}
\norm{U_2}_{L^\infty(B_{R/2})}
&\le C\left(\gamma^{1/\nu'}(U^\nu)^{1/\nu}_{B_R}+(U^2)^{1/2}_{B_R}+(F^2)^{1/2}_{B_R}\right),
\end{align}
where  $C=C(n,\lambda,\nu)$.
\item
Suppose that $(\bf{A3}\,(\gamma))$ $(a)$ and $(b)$ hold at $0\in \partial \Omega$ with $\gamma\in (0,1/2)$.
Then, for $R\in (0, R_0]$, $(\vec u,p)$ admits a decomposition 
\[
(\vec u,p)=(\vec u_1,p_1)+(\vec u_2,p_2) \quad \text{in }\,  \Omega_R,
\]
and we have 
\begin{align}		\label{123.eq1b}
(U_1^2)^{1/2}_{\Omega_R}&\le C\left(\gamma^{1/\nu'}(U^\nu)^{1/\nu}_{\Omega_R}+(F^2)^{1/2}_{\Omega_R}\right),\\
\label{123.eq1c}
\norm{U_2}_{L^\infty(\Omega_{R/4})}
&\le C\left(\gamma^{1/\nu'}(U^\nu)^{1/\nu}_{\Omega_R}+(U^2)^{1/2}_{\Omega_R}+(F^2)^{1/2}_{\Omega_R}\right),
\end{align}
where  $C=C(n,\lambda,\nu)$.
\end{enumerate}
Here, we define $U_i$ in the same way as $U$ with $p$ and $\vec u$ replaced by $p_i$ and $\vec u_i$, repectively.
\end{theorem}

\begin{proof}
The proof is an adaptation of that of \cite[Lemma 8.3]{MR2835999}.
To prove assertion $(i)$, 
we denote
\[
\sL_0\vec u=-D_\alpha(A_{\alpha\beta}^0D_\beta \vec u),
\]
where $A_{\alpha\beta}^0=(A_{\alpha\beta})_{B_R}$.
By Lemma \ref{122.lem1}, there exists a  unique solution $(\vec u_1,p_1)\in W^{1,2}_0(B_R)^n\times {L}_0^2(B_R)$ of the problem 
\begin{equation*}		
\left\{
\begin{aligned}
\sL_0\vec u_1+Dp_1=\vec f+D_\alpha \vec f_\alpha+D_\alpha\vec h_\alpha&\quad \text{in }\, B_R,\\
\dv \vec u_1=  g-( g)_{B_R}&\quad \text{in }\, B_R,
\end{aligned}
\right.
\end{equation*}
where 
\[
\vec h_\alpha=(A_{\alpha\beta}^0-A_{\alpha\beta})D_\beta \vec u.
\]
We also get from \eqref{122.eq1a} that  (recall $R\le R_0\le 1$)
\begin{equation*}
\norm{U_1}_{L^2(B_R)}\le C\left(\norm{\vec h_\alpha}_{L^2(B_R)}+\norm{F}_{L^2(B_R)}\right),
\end{equation*}
where $C=C(n,\lambda)$.
Therefore, by using the fact that 
\begin{equation*}		
\norm{\vec h_\alpha}_{L^2(B_R)}\le C\bignorm{A_{\alpha\beta}^0-A_{\alpha\beta}}_{L^1(B_R)}^{1/\nu'}\norm{D\vec u}_{L^\nu(B_R)}\le C_\nu\gamma^{1/\nu'}\abs{B_R}^{1/\nu'}\norm{D\vec u}_{L^\nu(B_R)},
\end{equation*}
we obtain \eqref{123.eq1}.
To see \eqref{123.eq1a}, we note  that $(\vec u_2,p_2)=(\vec u,p)-(\vec u_1,p_1)$ satisfies 
\begin{equation*}		
\left\{
\begin{aligned}
\sL_0\vec u_2+Dp_2=0&\quad \text{in }\, B_{R},\\
\dv \vec u_2= (g)_{B_R} &\quad \text{in }\, B_{R}.
\end{aligned}
\right.
\end{equation*}
Then by Lemma \ref{160924@lem1}, we get 
\[
\norm{U_2}_{L^\infty(B_{R/2})}\le C(U_2^2)^{1/2}_{B_R}+C(\abs{g}^2)^{1/2}_{B_R},
\]
and thus,  we conclude \eqref{123.eq1a} from \eqref{123.eq1}.

Next, we prove assertion $(ii)$.
Without loss of generality, we may assume that $(\bf{A3}(\gamma))$ $(b)$ hols at $0$ in the original coordinate system.
Define $\sL_0$ as above.
Let us fix $y:=(\gamma R,0,\ldots,0)$ and denote
\[
B_R^\gamma:=B_R\cap \set{x_1> \gamma R}.
\]
Then we have 
\begin{equation*}		
B_{R/2}\cap \set{x_1>\gamma R}\subset B_{R/2}^+(y)\subset B_R^\gamma.
\end{equation*}
Take a smooth function $\chi$ defined on $\bR$ such that 
\[
\chi(x_1)\equiv 0 \text{ for } x_1\le \gamma R, \quad \chi(x_1)\equiv 1 \text{ for } x_1\ge 2\gamma R, \quad \abs{\chi'}\le C(\gamma R)^{-1}.
\]
We then find that  $(\hat{\vec u}(x), \hat{p}(x))=(\chi(x_1) \vec u(x), \chi(x_1)p(x))$ satisfies 
\[
\left\{
\begin{aligned}
\sL_0 \hat{\vec u}+D\hat{p}=\vec \cF&\quad \text{in }\, B^\gamma_R,\\
\dv \hat{\vec u}=\cG &\quad \text{in }\, B_R^\gamma,\\
\hat{\vec u}=0 &\quad \text{on }\, B_R\cap \set{x_1=\gamma R},
\end{aligned}
\right.
\]
where we use the notation $\cG=D\chi\cdot \vec u+\chi g$ and
\begin{multline*}
\vec\cF=\chi \vec f+\chi D_\alpha \vec f_\alpha+p D\chi\\
+D_\alpha\big(A_{\alpha\beta}^0D_\beta ((1-\chi)\vec u)-(A_{\alpha\beta}^0-A_{\alpha\beta})D_\beta \vec u\big)+(\chi-1)D_\alpha(A_{\alpha\beta}D_\beta \vec u).
\end{multline*}
Let $(\hat{\vec u}_1,\hat{p}_1)\in W^{1,2}_0\big(B^+_{R/2}(y)\big)^n\times {L}^2_0\big(B^+_{R/2}(y)\big)$ satisfy 
\begin{equation}		\label{1128.eq0}
\left\{
\begin{aligned}
\sL_0 \hat{\vec u}_1+D\hat{p}_1=\vec \cF &\quad \text{in }\, B_{R/2}^+(y),\\
\dv \hat{\vec u}_1=\cG-(\cG)_{B^+_{R/2}(y)} &\quad \text{in }\, B_{R/2}^+(y),\\
\hat{\vec u}_1=0 &\quad \text{on }\, \partial B^+_{R/2}(y).
\end{aligned}
\right.
\end{equation}
Then by testing with $\hat{\vec u}_1$ in \eqref{1128.eq0}, we obtain 
\begin{align}
\nonumber		
&\int_{B^+_{R/2}(y)}A_{\alpha\beta}^0D_\beta \hat{\vec u}_1\cdot D_\alpha \hat{\vec u}_1\,dx\\
\nonumber
&=\int_{B^+_{R/2}(y)}\vec f\cdot (\chi \hat{\vec u}_1)-\vec f_\alpha \cdot D_\alpha(\chi \hat{\vec u}_1)+pD\chi\cdot \hat{\vec u}_1\,dx\\
\nonumber
&\quad +\int_{B^+_{R/2}(y)}-A_{\alpha\beta}^0D_\beta ((1-\chi)\vec u)\cdot D_\alpha \hat{\vec u}_1+(A_{\alpha\beta}^0-A_{\alpha\beta})D_\beta \vec u\cdot D_\alpha \hat{\vec u}_1\,dx\\
\label{1201.eq1}
&\quad +\int_{B^+_{R/2}(y)}-A_{\alpha\beta}D_\beta \vec u\cdot D_\alpha((\chi-1)\hat{\vec u}_1)\,dx+\hat{p}_1 (D\chi\cdot \vec u+\chi g)\,dx.
\end{align}
Note that
\[
\abs{D\chi(x_1)}+\abs{D(1-\chi(x_1))}\le C(x_1-\gamma R)^{-1}, \quad \forall x_1>\gamma R.
\]
Therefore, we obtain by Lemma \ref{0323.lem1} that  
\begin{equation}		\label{123.eq2}		
\norm{D(\chi \hat{\vec u}_1)}_{L^2(B^+_{R/2}(y))}+\norm{D((1-\chi) \hat{\vec u}_1)}_{L^2(B^+_{R/2}(y))}\le C\norm{D\hat{\vec u}_1}_{L^2(B^+_{R/2}(y))},
\end{equation}
and hence, we also have 
\begin{equation}		\label{123.eq2a}
\norm{D\chi \cdot \hat{\vec u}_1}_{L^2(B^+_{R/2}(y))}\le C\norm{D\hat{\vec u}_1}_{L^2(B^+_{R/2}(y))}.
\end{equation}
From \eqref{123.eq2a} and H\"older's inequality, we get 
\begin{align}
\nonumber
\int_{B^+_{R/2}(y)}pD\chi \cdot \hat{\vec u}_1\,dx&\le \norm{p}_{L^2(B^+_{R/2}(y)\cap \set{x_1<2\gamma R})}\norm{D\hat{\vec u}_1}_{L^2(B^+_{R/2}(y))}\\
\label{1201.eq3}
&\le C\gamma^{1/\nu'}R^{n/\nu'}\norm{p}_{L^\nu(\Omega_R)}\norm{D\hat{\vec u}_1}_{L^2(B^+_{R/2}(y))}.
\end{align}
Then, by applying \eqref{123.eq2}--\eqref{1201.eq3}, and the fact that  (recall $R\le R_0\le 1$)
\begin{equation*}		
\norm{\hat{\vec u}_1}_{L^2(B^+_{R/2}(y))}\le C(n)\norm{D\hat{\vec u}_1}_{L^2(B^+_{R/2}(y))}
\end{equation*}
to \eqref{1201.eq1}, 
we have   
\begin{equation*}		
\norm{D\hat{\vec u}_1}_{L^2(B^+_{R/2}(y))}\le \epsilon\norm{\hat{p}_1}_{L^2(B_{R/2}^+(y))}+C_\epsilon \norm{F}_{L^{2}(\Omega_R)}+C_\epsilon \cK, \quad \forall\epsilon>0,
\end{equation*}
where 
\begin{multline*}		
\cK:=\gamma^{1/\nu'}R^{n/\nu'}\norm{p}_{L^\nu(\Omega_R)}
+\norm{D((1-\chi)\vec u)}_{L^2(B^+_{R/2}(y))}\\
+\norm{D\vec u}_{L^2(B^+_{R/2}(y)\cap \set{x_1<2\gamma R })}+\bignorm{(A_{\alpha\beta}^0-A_{\alpha\beta})D_\beta \vec u}_{L^2(B^+_{R/2}(y))}.
\end{multline*}
Similarly, we have 
\begin{equation*}		
\norm{\hat{p}_1}_{L^2(B^+_{R/2}(y))}\le C\norm{D\hat{\vec u}_1}_{L^2(B^+_{R/2}(y))}+C\norm{F}_{L^2(\Omega_R)}+C\cK.
\end{equation*}
Therefore, from the above two inequality, we conclude that 
\begin{equation}		\label{1128.eq1a}
\norm{\hat{p}_1}_{L^2(B^+_{R/2}(y))}+\norm{D\hat{\vec u}_1}_{L^2(B^+_{R/2}(y))}\le C\norm{F}_{L^2(\Omega_R)}+C\cK,
\end{equation}
where $C=C(n,\lambda, \nu)$.
Now we claim that 
\begin{equation}		\label{123.eq3}
\cK\le C\gamma^{1/\nu'}R^{n/\nu'}\norm{U}_{L^\nu(\Omega_R)}.
\end{equation}
Observe that by H\"older's inequality and Lemma \ref{0812.lem1}, we have 
\begin{equation}		\label{1201.eq1a}
\norm{D\vec u}_{L^2(B^+_{R/2}(y)\cap \set{x_1<2\gamma R})} \le C(n,\nu)\gamma^{1/\nu'}R^{n/\nu'}\norm{D\vec u}_{L^\nu(\Omega_{R})}.
\end{equation}
We also have 
\begin{align}
\nonumber
\bignorm{(A_{\alpha\beta}^0-A_{\alpha\beta})D_\beta \vec u}_{L^2(B^+_{R/2}(y))}&\le C\left(\int_{B_{R}}\bigabs{A_{\alpha\beta}^0-A_{\alpha\beta}}\,dx\right)^{1/\nu'}\norm{D\vec u}_{L^\nu(B^+_{R/2}(y))}\\
\nonumber
&\le C\gamma^{1/\nu'}R^{n/\nu'}\norm{D\vec u}_{L^\nu(\Omega_R)},
\end{align}
where $C=C(n,\lambda,\nu)$.
To estimate $\norm{D((1-\chi)\vec u)}_{L^2(B^+_{R/2}(y))}$, we recall that $\chi-1=0$ for $x_1\ge 2\gamma R$.
For any $y'\in B'_{R}$, let $\hat{y}_1=\hat{y_1}(y')$ be the largest number such that $\hat{y}=(\hat{y}_1,y')\in \partial \Omega$.
Since $\abs{\hat{y}_1}\le \gamma R$, we have 
\[
x_1-\hat{y}_1\le x_1+\gamma R\le 3\gamma R , \quad \forall x_1\in [\gamma R, 2\gamma R],
\]
and thus, we obtain
\begin{equation*}		
\abs{D\chi(x_1)}\le C(x_1-\hat{y}_1), \quad \forall x_1\in [\gamma R, 2\gamma R]
\end{equation*}
Therefore, we find that 
\begin{align}
\nonumber
\int_{\gamma R}^r\abs{D((1-\chi)\vec u)(x_1,y')}^2\,dx_1&\le \int_{\hat{y}_1}^r\abs{D((1-\chi)\vec u)(x_1,y')}^2\,dx_1\\
\label{1201.eq1d}
&\le C\int_{\hat{y}_1}^r \abs{D\vec u(x_1,y')}^2\,dx_1,
\end{align}
where $r=r(y')=\min\big(2\gamma R, \sqrt{R^2-\abs{y'}^2}\big)$.
We then get from \eqref{1201.eq1d} that 
\begin{equation*}		
\norm{D((1-\chi)\vec u)}_{L^2(B^+_{R/2}(y))}\le C\gamma^{1/\nu'}R^{n/\nu'}\norm{D\vec u}_{L^\nu(\Omega_{R})},
\end{equation*}
where $C=C(n,\nu)$.
From the above estimates, we obtain \eqref{123.eq3}, and thus, by combining \eqref{1128.eq1a} and \eqref{123.eq3}, we conclude
\begin{equation}		\label{123.eq3a}
\norm{\hat{p}_1}_{L^2(B^+_{R/2}(y))}+\norm{D\hat{\vec u}_1}_{L^2(B^+_{R/2}(y))}\le C\left(\gamma^{1/\nu'}R^{n/\nu'}\norm{U}_{L^\nu(\Omega_R)}+\norm{F}_{L^2(\Omega_R)}\right),
\end{equation}
where $C=C(n,\lambda,\nu)$.

Now, we are ready to show the estimate \eqref{123.eq1b}.
We extend $\hat{\vec u}_1$ and $\hat{p}_1$ to be zero in $\Omega_R\setminus B^+_{R/2}(y)$.
Let $(\vec u_1,p_1)=\big(\hat{\vec u}_1+(1-\chi)\vec u, \hat{p}_1+(1-\chi)p\big)$. 
Since $(1-\chi)\vec u$ vanishes for $x_1\ge 2\gamma R$, by using the second inequality in \eqref{1201.eq1d} and H\"older's inequality as in \eqref{1201.eq1a}, we see that 
\begin{equation*}
\norm{D((1-\chi)\vec u)}_{L^2(B_{R/2})}\le C(n)\gamma^{1/\nu'}R^{n/\nu'}\norm{D\vec u}_{L^\nu(\Omega_R)}.
\end{equation*}
Moreover, it follows from H\"older's inequality that 
\begin{equation*}
\norm{(1-\chi)p}_{L^2(B_{R/2})}\le C(n)\gamma^{1/\nu'}R^{n/\nu'}\norm{p}_{L^\nu(\Omega_R)}.
\end{equation*}
Therefore, we conclude \eqref{123.eq1b} from \eqref{123.eq3a}.

Next, let us set $(\vec u_2,p_2)=(\vec u,p)-(\vec u_1,p_1)$.
Then, it is easily seen that $(\vec u_2,p_2)=(0,0)$ in $\Omega_{R}\setminus B^\gamma_{R}$ and $(\vec u_2,p_2)$ satisfies
\[
\left\{
\begin{aligned}
\sL_0\vec u_2+Dp_2=0 &\quad \text{in }\, B^+_{R/2}(y),\\
\dv \vec u_2=(\cG)_{B^+_{R/2}(y)}&\quad \text{in }\, B^+_{R/2}(y),\\
\vec u_2=0 &\quad \text{on }B_{R/2}(y)\cap \set{x_1=\gamma R}.
\end{aligned}
\right.
\]
By Lemma \ref{160924@lem1}, we get
\[
\norm{U_2}_{L^\infty(B^+_{R/2})}\le CR^{-n/2}\left(\norm{U_2}_{L^2(\Omega_R)}+\norm{\cG}_{L^2(B^+_{R/2}(y))}\right),
\]
and thus,  from \eqref{1201.eq1a} and \eqref{123.eq1b}, we obtain \eqref{123.eq1c}.
This completes the proof of the theorem.
\end{proof}

Now, we recall the maximal function theorem.
Let 
\[
\sB=\set{B_r(x):x\in \bR^n,\, r\in (0,\infty)}.
\]
For a function $f$ on a set $\Omega\subset \bR^n$, we define its maximal function $\cM(f)$ by 
\[
\cM(f)(x)=\sup_{B\in \sB,\, x\in B}\fint_B \abs{f(y)}1_\Omega\,dy.
\]
Then for $f\in L^q(\Omega)$ with $1<q\le \infty$, we have 
\begin{equation*}		
\norm{\cM(f)}_{L^q(\bR^n)}\le C\norm{f}_{L^q(\Omega)},
\end{equation*}
where $C=C(n,q)$.
As is well known, the above inequality  is due to the Hardy-Littlewood maximal function theorem.
Hereafter, we use the notation
\begin{equation*}				
\begin{aligned}
\cA(s)&=\set{x\in \Omega: U(x)>s},\\
\cB(s)&=\bigset{x\in \Omega:\gamma^{-1/\nu'}(\cM(F^2)(x))^{1/2}+(\cM(U^\nu)(x))^{1/\nu}>s}.
\end{aligned}
\end{equation*}
With Theorem \ref{123.thm1} in hand, we get the following corollary.

\begin{corollary}		\label{0126.cor1}
Suppose that  $(\bf{A3}\, (\gamma))$ holds with $\gamma\in (0,1/2)$, and $0\in \overline{\Omega}$.
Let  $2<\nu<q<\infty$ and  $\nu'=2\nu/(\nu-2)$.
Assume $(\vec u,p)\in W^{1,q}_0(\Omega)^n\times {L}^q_0(\Omega)$ satisfies 
\begin{equation*}		
\left\{
\begin{aligned}
\sL\vec u+Dp=D_\alpha\vec f_\alpha+\vec f&\quad \text{in }\, \Omega,\\
\dv \vec u= g &\quad \text{in }\, \Omega,
\end{aligned}
\right.
\end{equation*}
where $\vec f,\, \vec f_\alpha\in L^2(\Omega)^n$ and $g\in L^2_0(\Omega)$.
Then there exists a constant $\kappa=\kappa(n,\lambda,\nu)>1$ such that the following holds: 
If 
\begin{equation}		\label{1204.eq1b}
\abs{\Omega_{R/32}\cap \cA(\kappa s)}\ge \gamma^{2/\nu'}\abs{\Omega_{R/32}}, \quad  R\in (0,R_0], \quad s>0,
\end{equation}
then we have 
\[
\Omega_{R/32}\subset \cB(s).
\]
\end{corollary}

\begin{proof}
By dividing $U$ and $F$ by $s$, we may assume $s=1$.
We prove by contradiction.
Suppose that there exists a point $x\in \Omega_{R/32}=B_{R/32}(0)\cap \Omega$ such that 
\begin{equation}		\label{1203.eq1}
\gamma^{-1/\nu'}(\cM(F^2)(x))^{1/2}+(\cM(U^\nu)(x))^{1/\nu}\le 1.
\end{equation}
In the case when $\dist(0,\partial \Omega)\ge R/8$, we note that 
\[
x\in B_{R/32}\subset B_{R/8}\subset \Omega.
\]
Due to Theorem \ref{123.thm1} $(i)$, we can decompose $(\vec u,p)=(\vec u_1,p_1)+(\vec u_2,p_2)$ in $B_{R/8}$ and then, by \eqref{1203.eq1}, we have 
\begin{equation*}	
(U_1^2)^{1/2}_{B_{R/8}}\le C_0\big(\gamma^{1/\nu'}(U^\nu)^{1/\nu}_{B_{R/8}}+(F^2)^{1/2}_{B_{R/8}}\big)\le C_0\gamma^{1/\nu'}
\end{equation*}
and 
\begin{equation*}		
\norm{U_2}_{L^\infty(B_{R/32})}\le C_0\big(\gamma^{1/\nu'}(U^\nu)^{1/\nu}_{B_{R/8}}+(U^2)^{1/2}_{B_{R/8}}+(F^2)^{1/2}_{B_{R/8}}\big)\le C_0,
\end{equation*}
where $C_0=C_0(n,\lambda,\nu)$.
From these inequalities and Chebyshev's inequality, we get 
\begin{align}
\nonumber
\bigabs{B_{R/32}\cap \cA(\kappa)}&=\bigabs{\set{x\in B_{R/32}:U(x)>\kappa}}\\
\label{126.eq1a}
&\le \bigabs{\set{x\in B_{R/32}:U_1>\kappa-C_0}}\le C(n)\frac{C_0^2}{(\kappa-C_0)^2}\gamma^{2/\nu'}\abs{B_{R/32}},
\end{align}
which contradicts with \eqref{1204.eq1b} if we choose $\kappa$ sufficiently large.

We now consider the case $\dist(0,\partial \Omega)<R/8$.
Let $y\in \partial \Omega$ satisfy $\abs{y}=\dist(0,\partial \Omega)$.
Then we have 
\[
x\in \Omega_{R/32}\subset \Omega_{R/4}(y).
\]
By Theorem \ref{123.thm1} $(ii)$,  we can decompose $(\vec u,p)=(\vec u_1,p_1)+(\vec u_2,p_2)$ in $\Omega_{R}(y)$ and then, by \eqref{1203.eq1}, we have 
\begin{equation*}
(U_1^2)^{1/2}_{\Omega_{R}(y)}\le C_0\gamma^{1/\nu'} \quad \text{and}\quad 
\norm{U_2}_{L^\infty(\Omega_{R/4}(y))}\le C_0.
\end{equation*}
From this, and by following the same steps used in deriving \eqref{126.eq1a}, we get 
\[
\bigabs{\Omega_{R/32}\cap \cA(\kappa)}\le C(n)\frac{C_0^2}{(\kappa-C_0)^2}\gamma^{2/\nu'}\abs{\Omega_{R/32}},
\]
which contradicts with \eqref{1204.eq1b} if we choose $\kappa$ sufficiently large.
\end{proof}

\subsection{Proof of Theorem \ref{0123.thm1}}
We fix $2<\nu<q$ and denote $\nu'=2\nu/(\nu-2)$.
Let $\gamma\in (0,1/2)$ be a constant to be chosen later and  $\kappa=\kappa(n,\lambda,\nu)$ be the constant in Corollary \ref{0126.cor1}.
Since
\[
\abs{\cA(\kappa s)}\le C_0(\kappa s)^{-1}\norm{U}_{L^2(\Omega)}
\]
for all $s>0$, where $C_0=C_0(n,K_0)$, 
 we get 
\begin{equation}		\label{1215.eq1}
\abs{\cA(\kappa s)}\le \gamma^{2/\nu'}\abs{B_{R_0/32}},
\end{equation}
provided that 
\[
s\ge \frac{C_0}{\kappa \gamma^{2/\nu'}|B_{R_0/32}|}\norm{U}_{L^2(\Omega)}:=s_0.
\]
Therefore, from \eqref{1215.eq1}, Corollary \ref{0126.cor1}, and Lemma \ref{0324.lem1}, we have the following upper bound of the distribution of $U$; 
$$
\abs{\cA(\kappa s)}\le C_1\gamma^{2/\nu'}\abs{\cB(s)} \quad \forall s>s_0,
$$
where $C_1=C_1(n)$.
Using this together with the fact that 
$$
\abs{\cA(\kappa s)}\le (\kappa s)^{-2}\norm{U}_{L^2(\Omega)}^2, \quad \forall s>0,
$$
we have 
\begin{align*}
\norm{U}_{L^q(\Omega)}^q&=q\int_0^\infty\abs{\cA(s)}s^{q-1}\,ds=q\kappa^{q}\int_0^\infty \abs{\cA(\kappa s)}s^{q-1}\,ds\\
&= q\kappa^q\int_0^{s_0}\abs{\cA(\kappa s)}s^{q-1}\,ds+q\kappa^q\int_{s_0}^\infty \abs{\cA(\kappa s)}s^{q-1}\,ds\\
&\le C_2\gamma^{2(2-q)/\nu'}\norm{U}_{L^2(\Omega)}^q+C_3\gamma^{2/\nu'}\int_0^\infty \abs{\cB(s)}s^{q-1}\,ds,
\end{align*}
where $C_2=C_2(n,\lambda,K_0,q,R_0)$ and $C_3=C_3(n,\lambda,q)$.
The Hardy-Littlewood maximal function theorem implies that 
$$
\norm{U}^q_{L^q(\Omega)}\le C_2\gamma^{2(2-q)/\nu'}\norm{U}_{L^2(\Omega)}^q+C_4\gamma^{(2-q)/\nu'}\norm{F}^q_{L^q(\Omega)}+C_4\gamma^{2/\nu'}\norm{U}_{L^q(\Omega)}^q,
$$
where $C_4=C_4(n,\lambda,q)$.
Notice from Lemma \ref{122.lem1} and H\"older's inequality that 
$$
\norm{U}_{L^2(\Omega)}^q\le  C_5\norm{F}_{L^q(\Omega)}^q,
$$
where $C_5=C_5(n,\lambda,K_0,q,A)$.
Combining the above two estimates and taking $\gamma=\gamma(n,\lambda,q)\in (0,1/2)$ sufficiently small, we conclude \eqref{1204.eq2}.
\hfill\qedsymbol

\section{Appendix}		\label{app}
In this section, we provide some lemmas.

\begin{lemma}		\label{0323.lem1}
Let $f\in W^{1,2}_0(I)$, where $I=(0, R)$.
Then we have
\begin{equation}		\label{0508.eq1}
\norm{x^{-1}f(x)}_{L_2(I)}\le C\norm{Df}_{L_2(I)},
\end{equation}
where  $C>0$ is a constant.
\end{lemma}

\begin{proof}
We first note that \eqref{0508.eq1} holds for any $f\in C^\infty([0,R])$ satisfying $Df(0)=0$;
see \cite[Lemma 7.9]{MR2835999}.
Suppose that $f\in W^{1,2}_0(I)$ and  $\set{f_n}$ is a sequence in $C^\infty_0([0,R])$ such that $f_n\to f$ in $W^{1,2}(I)$.
Then by the Sobolev embedding theorem, $f_n\to f$ in $C([0,R])$.
Since the estimates \eqref{0508.eq1} is valid for $f_n$, we obtain by Fatou's lemma that 
\begin{align*}
\int_0^R\bigabs{x^{-1}f(x)}^2\,dx&=\int_0^R\lim_{n\to \infty}\bigabs{x^{-1}f_n(x)}^2\,dx\\
&\le \liminf_{n\to \infty}\int_0^R\bigabs{x^{-1}f_n(x)}^2\,dx\\
&\le C\liminf_{n\to \infty}\int_0^R\abs{Df_n(x)}^2\,dx=\int_0^R\abs{Df(x)}^2\,dx,
\end{align*}
which establishes \eqref{0508.eq1}.
\end{proof}

\begin{lemma}		\label{0812.lem1}
Suppose that $(\bf{A3}(\gamma))$ $(b)$ holds at $0\in \partial \Omega$ with $\gamma\in \big(0,\frac{1}{2}\big)$.
Then for $R\in (0,R_0]$, we have 
\begin{equation}		\label{123.a1}		
\abs{\Omega_R}\ge CR^n,
\end{equation}
and
\begin{equation}		\label{123.a1a}		
\abs{\Omega_R\cap \set{x:x_1<2\gamma R}}\le C\gamma\abs{\Omega_R},
\end{equation}
where $C=C(n)$.
\end{lemma}

\begin{proof}
Note that 
\begin{equation}		\label{0410.eq1b}
\abs{\Omega_R\cap \set{x:x_1<2\gamma R}}\le 2^n\gamma R^n.
\end{equation}
Let us fix $a\in(\frac{1}{2},1)$ and 
\[
\cQ=\Set{x: \abs{x_1}<aR,\, \abs{x_i}<\sqrt{\frac{1-a^2}{d-1}}R, \, i=2,\ldots,n}. 
\]
Then we have
\[
\cQ\cap \set{x:x_1>R/2}\subset \Omega_R,
\]
and hence, we obtain
\begin{equation*}		
\left(a-\frac{1}{2}\right)\left(\frac{1-a^2}{n-1}\right)^{(n-1)/2}R^n=\bigabs{\cQ\cap \set{x:x_1>R/2}}\le \abs{\Omega_R},
\end{equation*}
which implies \eqref{123.a1}.
By combining  \eqref{123.a1} and \eqref{0410.eq1b}, we get \eqref{123.a1a}.
\end{proof}

The following lemma is a result from the measure theory on the ``crawling of ink spots" which can be found in \cite{MR0563790, MR0579490}. 
See also \cite{MR2069724}.

\begin{lemma}		\label{0324.lem1}
Suppose that $(\bf{A3}(\gamma))$ $(b)$ holds with $\gamma\in \big(0, \frac{1}{2}\big)$.
Let $A$ and $B$ are measurable sets satisfying $A\subset B\subset \Omega$, and that there exists a constant $\epsilon\in (0,1)$ such that the following hold:
\begin{enumerate}[(i)]
\item
$\abs{A}<\epsilon \abs{B_{R_0/32}}$.
\item
For any $x\in \overline{\Omega}$ and for all $R\in (0, R_0/32]$ with $\abs{B_R(x)\cap A}\ge \epsilon \abs{B_R}$, we have $\Omega_R(x)\subset B$.
\end{enumerate}
Then we get 
\[
\abs{A}\le C\epsilon \abs{B},
\]
where $C=C(n)$.
\end{lemma}

\begin{proof}
We first claim that for a.e. $x\in A$, there exists  $R_x\in (0, R_0/32)$ such that 
\begin{equation*}
\abs{A\cap B_{R_x}(x)}=\epsilon \abs{B_{R_x}}
\end{equation*}
and 
\begin{equation}		\label{0324.eq1a}
\abs{A\cap B_{R}(x)}< \epsilon \abs{B_R}, \quad \forall R\in (R_x, R_0/32].
\end{equation}
Note that the function $\rho=\rho(r)$ given by 
\[
\rho(r)=\frac{\abs{A\cap B_r(x)}}{\abs{B_r}}=\fint_{B_r(x)}1_A(y)\,dy
\]
is continuous on $[0, R_0]$.
Since $\rho(0)=1$ and $\rho(R_0/32)<\epsilon$, there exists $r_x\in (0, R_0/32)$ such that $\rho(r_x)=\epsilon$.
Then we get the claim by setting
\[
R_x:=\max\set{r_x\in (0, R_0):\rho(r_x)=\epsilon}.
\]
Hereafter, we denote by
$$
\cU=\set{B_{R_x}(x):x\in A'},
$$
where $A'$ is the set of all points $x\in A$ such that $r_x$ exists.
Then by the Vitali lemma, we have a countable subcollection $G$ such that  
\begin{enumerate}[(a)]
\item
$Q\cap Q'=\emptyset$ for any $Q, Q'\in G$ satisfying $Q\neq Q'$.
\item 
$A'\subset \cup \set{B_{5R}(x): B_R(x)\in G}$.
\item
$\abs{A}=|A'|\le 5^{n}\sum_{Q\in G}\abs{Q}$.
\end{enumerate}
By the assumption (i) and \eqref{0324.eq1a}, we see that 
$$
\abs{A\cap B_{5R}(x)}<\epsilon\abs{B_{5R}}=\epsilon5^n\abs{B_R}, \quad \forall B_R(x)\in G.
$$
Using this together with the assumption (ii) and Lemma \ref{0812.lem1}, we have 
\begin{align*}
\abs{A}&=\bigabs{\cup\set{B_{5R}(x)\cap A:B_{R}(x)\in G}}\le \sum_{B_{R}(x)\in G}\abs{B_{5R}(x)\cap A}\\
&<\epsilon 5^{n}\sum_{B_R(x)\in G}\abs{B_R(x)}\le \epsilon C(n)\sum_{B_R(x)\in G}\abs{B_R(x)\cap \Omega}\\
&=\epsilon C(n)\bigabs{\cup\set{B_R(x)\cap \Omega:B_R(x)\in G}}\\
&\le \epsilon C(n)\abs{B},
\end{align*}
which completes the proof.
\end{proof}

\begin{acknowledgment}
The authors would like to express their sincerely gratitude to the referee for careful reading and for many helpful comments and suggestions.
The authors also thank Doyoon Kim for valuable discussions and comments.
Ki-Ahm Lee was supported by the National Research Foundation of Korea (NRF)
grant funded by the Korea government (MSIP) (No. 2014R1A2A2A01004618). Ki-Ahm
Lee also hold a joint appointment with the Research Institute of Mathematics of Seoul
National University.
Jongkeun Choi was supported by BK21 PLUS SNU Mathematical Sciences Division.
\end{acknowledgment}

\bibliographystyle{plain}





\end{document}